\documentclass[leqno,10pt]{amsart}

\usepackage[latin1]{inputenc}
\usepackage[T1]{fontenc}
\usepackage[english]{babel}
\usepackage{amsmath}
\usepackage{verbatim}
\usepackage{amsfonts}
\usepackage[all]{xy}
\usepackage{graphicx}
\usepackage{amssymb}
\usepackage{MnSymbol}
\usepackage{amsthm}
\usepackage{hyperref}
\usepackage{color}

\newcommand{\isom}{\ \cong\ }

\newcommand{\Sets}{\mathrm{Sets}}
\DeclareMathOperator{\rk}{rk}

\newcommand{\leftexp}[2]{{\vphantom{#2}}^{#1}{#2}}
\newcommand{\incl}[1][r]{\ar@<-0.2pc>@{^(-}[#1] \ar@<+0.2pc>@{-}[#1]}

\newcommand{\ZZ}{\mathcal{Z}}
\newcommand{\UU}{\mathcal{U}}
\newcommand{\XX}{{\mathfrak{X}}}

\newcommand{\weight}{\mathrm{weight}}
\newcommand{\Card}{\mathrm{Card}}

\renewcommand{\AA}{\mathcal{A}}

\newcommand{\OO}{\mathcal{O}}
\newcommand{\sO}{\mathcal{O}}
\newcommand{\HH}{\mathcal{H}}

\newcommand{\HHm}{\mathcal{H}^{\mathrm{main}}}
\newcommand{\BB}{\mathcal{B}}
\newcommand{\OH}{\widehat{\OO}_{\HH,[X]}}

\newcommand{\NN}{\mathbb{N}}
\newcommand{\Coker}{{\mathrm{Coker}\,}}
 
\newcommand{\Ach}{\widehat{A}}

\newcommand{\Shut}{{\widehat{S}}}
\newcommand{\Rch}{\widehat{R}}
\newcommand{\Gm}{\mathbb{G}_m}
\newcommand{\Aff}{{\mathbb{A}}_{k}^{1}}

\newcommand{\LL}{\mathcal{L}}
\newcommand{\RR}{\mathcal{R}}
\newcommand{\PP}{\mathbb{P}}
\newcommand{\ZZZ}{\mathbb{Z}}
\newcommand{\AAA}{\mathbb{A}}

\newcommand{\mm}{\mathfrak{m}}

\newcommand{\ohne}{{\ \setminus \ }}
\newcommand{\Sym}{\mathrm{Sym}}
\newcommand{\Irr}{\mathrm{Irr}}
\newcommand{\Tor}{\mathrm{Tor}}
\DeclareMathOperator{\into}{{\hookrightarrow}}
\DeclareMathOperator{\Art}{{Art}}

\newcommand{\Noeth}{{\widehat \Art}}
\DeclareMathOperator{\ob}{{ob}}
\DeclareMathOperator{\id}{{id}}
\newcommand{\Sing}{\mathrm{Sing}}

\newcommand{\Hom}{\mathrm{Hom}}
\newcommand{\Ext}{\mathrm{Ext}}

\newcommand{\Ker}{\mathrm{Ker}}
\renewcommand{\Im}{\mathrm{Im}}

\newcommand{\Aut}{\mathrm{Aut}}
\newcommand{\Stab}{\mathrm{Stab}}

\newcommand{\Hilb}{\mathrm{Hilb}}
\newcommand{\rg}{\mathrm{rk}}
\newcommand{\Spec}{\mathrm{Spec}}
\renewcommand{\to}[1][]{\xrightarrow{\ #1\ }}

\newcommand{\gothq}{{\mathfrak{q}}}
\newcommand{\sss}{{\mathfrak{s}}}
\newcommand{\vv}{{\mathfrak{v}}}
\newcommand{\gotho}{{\mathfrak{o}}}

\newtheoremstyle{citing}
  {}
  {}
  {\itshape}
  {}
  {\bfseries}
  {\textbf{.}}
  {.5em}
  {\thmnote{#3}}

\theoremstyle{plain}
\newtheorem{theorem}{Theorem}[section]
\newtheorem{lemma}[theorem]{Lemma}
\newtheorem{proposition}[theorem]{Proposition}
\newtheorem{corollary}[theorem]{Corollary}
\newtheorem{hypo}[theorem]{Hypothesis}
\newtheorem{algo}[theorem]{Algorithm}

\theoremstyle{definition}
\newtheorem{definition}[theorem]{Definition}

\newtheorem{notation}[theorem]{Notation}

\theoremstyle{remark}
\newtheorem{remark}[theorem]{Remark}

{\theoremstyle{citing}
\newtheorem*{custom}{}}

\begin{document}

\title[Invariant deformation theory]{Invariant deformation theory of affine schemes with reductive group action} 

\author{Christian Lehn}
\address{Christian Lehn\\Institut f\"ur Algebraische Geometrie\\Gottfried Wilhelm Leibniz Universit\"at Hannover\\ 
Welfengarten 1\\30167 Hannover\\Germany}
\email{lehn@math.uni-hannover.de}

\author{Ronan Terpereau}
\address{Ronan Terpereau\\Institut f\"ur Mathematik\\ Johannes Gutenberg-Universit\"at Mainz\\
Staudingerweg 9\\55099 Mainz\\Germany}
\email{rterpere@uni-mainz.de}

\begin{abstract}
We develop an invariant deformation theory, in a form accessible to practice, for affine schemes $W$ equipped with an action of a reductive algebraic group $G$. Given the defining equations of a $G$-invariant subscheme $X \subset W$, we device an algorithm to compute the universal deformation of $X$ in terms of generators and relations up to a given order. In many situations, our algorithm even computes an algebraization of the universal deformation. As an application, we determine new families of examples of the invariant Hilbert scheme of Alexeev and Brion, where $G$ is a classical group acting on a classical representation, and we describe their singularities.
\end{abstract}

\maketitle

\enlargethispage*{2\baselineskip}

\setcounter{tocdepth}{1}

\tableofcontents

\thispagestyle{empty}

\section{Introduction}\label{sec intro}

Let $k$ be a fixed algebraically closed field of characteristic zero. Let us fix a reductive algebraic group $G$, an affine $G$-scheme $W$ of finite type over $k$, and a Hilbert function $h: \Irr(G) \to \NN$ which assigns to every irreducible representation of $G$ a nonnegative integer. We denote by $\HH:=\Hilb_{h}^G(W)$ the invariant Hilbert scheme of Alexeev and Brion \cite{AB} corresponding to the triple $(G,W,h)$; see section \ref{sec IHS} for more details.

This is a quite thrilling and somewhat mysterious object, and there is a large number of articles dedicated to its study. It has rendered services to the classification theory of spherical varieties; see \cite[\S4]{Br} for an overview and further references. Moreover, in many cases the invariant Hilbert scheme furnishes a canonical candidate for a resolution of singularities of the categorical quotient $W/\!/G=\Spec(k[W]^G)$. Indeed, if we take $h=h_0$ the Hilbert function of the general fibers of the quotient morphism $W \to W/\!/G$ (see section \ref{sec IHS} for the precise definition of $h_0$), then we have the so-called \textit{Hilbert-Chow morphism} $\gamma:\HH\to W/\!/G$, which is a projective morphism that maps a unique irreducible component $\HHm$ of $\HH$, the so-called \textit{main component}, birationally to $W/\!/G$. Examples where $\gamma$ is a resolution can be found in \cite{IN1,IN2,BKR,LS} for finite groups, and in \cite{JaRe,Tanja2,Ter1,Ter2} for classical groups.

On the other side, it is difficult to get control about the invariant Hilbert scheme in a hands-on way. It has been described only in some very special situations where $\HH$ was generally first shown to be smooth by some ad hoc arguments. Examples where $h=h_0$ can be found in the references mentioned above, and examples where $h$ takes values in $\{0,1\}$ can be found in \cite{J,BC08,PaBart,CF}. However, when $\HH$ is singular, explicit descriptions of examples as well as a general strategy to describe the singularities were missing so far.

The goal of this article is to describe an algorithm which provides Zariski-local equations of $\HH$ at an arbitrary point; see section \ref{section algorithm} and in particular Algorithm \ref{algo main}. This is, in some sense, the strongest form of information that one can have about a scheme. To achieve this, we developed an invariant deformation theory in a form accessible to practice; see section \ref{sec defo}. Our algorithm is completely general and can be applied to any point $[X] \in \HH$ as soon as there is an action of the multiplicative group $\Gm$ on $W$ by $G$-equivariant automorphisms with strictly positive weights on $k[W]$ and on the cotangent space $(T_{[X]}\HH)^\vee$; see Hypothesis \ref{hypo weights}.

As an illustration, we apply our algorithm in three situations: 
\begin{enumerate}
\item \label{cas1}$G=GL_3$ acting on $W=(k^3)^{\oplus n_1}\oplus (k^{3*})^{\oplus n_2}$, which is the sum of $n_1$ copies of the defining representation, and $n_2$ copies of its dual (section \ref{casGL3});
\item \label{cas2} $G=SO_3$ acting on $W=(k^3)^{\oplus 3}$ (section \ref{HclassSOngeneralset}); and
\item \label{cas3} $G=O_3$ acting on $W=(k^3)^{\oplus n}$ (section \ref{casO3}).
\end{enumerate}

\begin{custom}[Theorems \ref{casSO3} and \ref{casGL3}]
Let $G$ and $W$ be as in situation \ref{cas1} or \ref{cas2}, and let $h=h_0$ be the Hilbert function of the general fibers of the quotient morphism $\nu:W \to W/\!/G$. 
Then the main component of the invariant Hilbert scheme is smooth, and thus the Hilbert-Chow morphism $\gamma:\HHm\to W/\!/G$ is a resolution of singularities. 
In both cases, $\HH$ is reduced, connected, and the union of two irreducible components $\HHm \cup  \HH'$, where $\HH'$ is smooth in situation \ref{cas1} but singular in situation \ref{cas2}. 
\end{custom}

Moreover, we give a description of the special fiber $\gamma^{-1}(\nu(0))$, both as an abstract scheme as well as in terms of the $G$-stable ideals that it parametrizes. Even in the case where we do not succeed to describe the invariant Hilbert scheme completely, our algorithm proves helpful. 
We obtain the following information in situation \ref{cas3}:

\begin{custom}[Theorem \ref{casO3}]
Let $G$ and $W$ be as in situation \ref{cas3}, and let $h=h_0$ be the Hilbert function of the general fibers of the quotient morphism $W \to W/\!/G$. 
Then the invariant Hilbert scheme $\HH$ is connected and has at least two irreducible components. 
\end{custom}

We give more detailed formulations of these results in the sections \ref{HclassSOngeneral} and \ref{two others}.
These examples have entered the focus by the work \cite{Terp} because they were the first examples with classical groups acting on classical representations where the invariant Hilbert scheme was known to be singular. However, this was shown simply by calculating the dimension of the tangent space and comparing it to the dimension of $W/\!/G$. Thus, no further properties of $\HH$ such as reducibility, or the smoothness of the main component $\HHm$, were given. Our results also show that the geometry of the invariant Hilbert scheme can be very diverse. Thus, it is important to have many more examples, and our algorithm gives a powerful tool to calculate them.

Let us explain how we obtain these results. The strategy is to \emph{localize} geometric properties of $\HH$ in special points and then to compute local equations at these points via Algorithm \ref{algo main}. The question is: which points in $\HH$ contain most information and how do we find them?
In situations \ref{cas1} to \ref{cas3}, there is each time an algebraic subgroup $H \subset \Aut^G(W)$ acting on $\HH$. Imagine that we want to study the singular locus $\HH^\Sing\subset \HH$ for example. Clearly,  $\HH^\Sing$ is closed and $H$-stable. As $\gamma:\HH\to W/\!/G$ is projective, also $\gamma(\HH^\Sing)$  is closed and $H$-stable, hence it contains a closed $H$-orbit. In situations \ref{cas1} to \ref{cas3}, there is only one closed orbit for the $H$-action on $W/\!/G$, and it consists in a single point, say $\{o\}$. The crucial argument is that now Borel's fixed point theorem implies the existence of fixed points in the projective scheme $\gamma^{-1}(o)\cap \HH^\Sing$ for any Borel subgroup $B \subset H$. This technique of localizing Borel fixed points of $\HH$ obviously applies to many other geometric properties such as non-reducedness or reducibility. Moreover, these fixed points can be calculated by hand in many examples. In our case, there are two fixed points each in situations \ref{cas2} and \ref{cas3}, and a single one in situation \ref{cas1}. This stipulates the following general strategy: 

\begin{enumerate}
\item\label{find fp} Find the fixed points for the $B$-action on $\HH$.
\item\label{connect fp} Try to connect these fixed points to the main component $\HHm$ in order to show connectedness.
\item \label{tgt} Determine the tangent space of $\HH$ at these fixed points.
\item \label{step33} For each of these fixed points, look for subgroups of $B$ isomorphic to $\Gm$ and satisfying Hypothesis \ref{hypo weights} . (There are such groups for all fixed points in situations \ref{cas1} to \ref{cas3}, except for one fixed point in situation \ref{cas3}.) 
\item\label{calculate loc ring at fp} For every fixed point satisfying Hypothesis \ref{hypo weights}, calculate an affine open neighborhood in $\HH$ thanks to Algorithm \ref{algo main}.
\end{enumerate}

Steps \ref{find fp}, \ref{connect fp}, and \ref{step33} are done by hand or by ad hoc methods for each example, Steps \ref{tgt} and \ref{calculate loc ring at fp} are done with a computer algebra system. A documented MACAULAY2 file containing all the computations in situations \ref{cas1} to \ref{cas3} is available on the webpage of the authors.

Equivariant deformation theory has been studied by Rim in \cite{R}. This includes the case where $G$ acts non-trivially on the base space of the versal deformation. The difference to our work is that firstly we study embedded and not abstract deformations, and secondly here we study invariant deformation theory, that is, deformations where $G$ acts trivially on the base space. Moreover, in our case due to the existence of the invariant Hilbert scheme, it is unnecessary to assume the existence of a versal deformation of the underlying variety without group action (and in fact a versal deformation does not exist in our examples).

Let us also mention that, when $G$ is a finite group and $X \subset W$ is a finite subscheme, our algorithm seems to be folklore. In particular, if $G$ is trivial, then the invariant Hilbert scheme $\HH$ is a punctual Hilbert scheme; the study of this latter via deformation theory can be found in the first chapters of \cite{Stevens} for instance. An example where $G \subset SL_2$ is the binary tetrahedral group can be found in \cite[\S2]{LS}. However, even in the particular setting where $G$ is a (non-trivial) finite group, we did not find any explanation for the validity of the algorithm. That is why we include a full treatment there.

The text is structured as follows. In section \ref{sec generalities} we recall some basics about the invariant Hilbert scheme and the Reynolds operator. Section \ref{sec defo} is the heart of this paper. There we develop an invariant deformation theory, which we present in a constructive way. Our algorithm can be deduced rather immediately from the presentation of the theoretical results. We summarize it in a completely algorithmic fashion in section \ref{section algorithm}. In section \ref{subsec gm action} we add Hypothesis \ref{hypo weights} on the positivity of weights, and deduce its theoretical consequences. In particular, there we give the argument for the stop condition of our algorithm, which is formulated in an algorithmic way in section \ref{subsec stop}. Finally, sections \ref{HclassSOngeneral} and \ref{two others} present the applications in situations \ref{cas1} to \ref{cas3}.\\

\noindent \textbf{Acknowledgments.} We are grateful to Manfred Lehn for explaining the idea of the algorithm in the case where $G$ is a finite group acting on a finite scheme. We would like also to thank Michel Brion, Edoardo Sernesi, and the referee for their careful reading of this paper and their valuable comments.

The first-named author gratefully acknowledges the support of the DFG through the research grant Le 3093/1-1 and the kind hospitality of the IMJ, Paris.
While working on this project, the second-named author benefited from the support of the DFG via the SFB/TR 45 ``Periods, moduli spaces and arithmetic of algebraic varieties''.

\section{Some background}  \label{sec generalities}

The aim of this section is to provide the reader with some definitions and basic results concerning the invariant Hilbert scheme, constructed by Alexeev and Brion in \cite{AB}, and about the Reynolds operator. The survey \cite{Br} gives a detailed introduction to the invariant Hilbert schemes. 

All schemes we consider are supposed to be separated and of finite type over an algebraically closed field $k$ of characteristic zero. We refer to \cite{Bo} for the necessary background material on algebraic groups.

\subsection{Generalities on the invariant Hilbert scheme} \label{sec IHS}
Let $G$ be a reductive group, let $S$ be a scheme, let $\ZZ$ be a $G$-scheme, and let $\pi: \ZZ \to S$ be a $G$-invariant affine morphism of finite type. 
The sheaf $\AA:=\pi_{*} \OO_{\ZZ}$ is a finitely generated $\OO_S$-algebra, and so is the sheaf of invariants $\AA^G$ by the Hilbert-Nagata Theorem (see e.g. \cite[Theorem 1.24]{Br10}). Denote by $\Irr(G)$ the set of isomorphism classes of irreducible $G$-modules. For any $M \in \Irr(G)$, we consider the sheaf of covariants $\AA_{(M)}:=\Hom^{G}(M,\AA)$ on $S$. By \cite[Lemma 2.1]{Br10}, each $\AA_{(M)}$ is a finitely generated $\AA^G$-module. 

Consider the evaluation map $\AA_{(M)}\otimes M \to \AA$. According to \cite[\S2.3]{Br}, the direct sum of all this evaluations gives a decomposition of the sheaf $\AA$ as a $(\OO_S,G)$-module: 
\begin{equation} \label{eq1}
\AA \cong \bigoplus_{M \in \Irr(G)} \AA_{(M)} \otimes M.
\end{equation}
Hereby, the $(\OO_S,G)$-module structure on each piece $ \AA_{(M)} \otimes M$ is given as follows: $G$ acts only on the factor $M$, and the $\OO_S$-module structure is induced by that of $\AA_{(M)}$. 

If $\AA^G$ is a coherent $\OO_S$-module and $\pi$ is flat, then each $\OO_S$-module $\AA_{(M)}$ is flat by \eqref{eq1}. Recall that flatness is equivalent to local freeness for a finitely generated module over a Noetherian ring. We get that each $\OO_S$-module $\AA_{(M)}$ is locally free of finite rank, the latter being constant on each connected component of $S$. Suppose that this rank is constant on $S$; then the function 
\[
h_{\ZZ}:\Irr(G) \to \NN, \qquad M\mapsto \rk \AA_{(M)}
\]
is called  the \emph{Hilbert function} of the family $\ZZ \to S$.

\begin{definition}  \label{Hilbert functor}
Let $h:\Irr(G) \to \NN$ be a function, and let $W$ be an affine $G$-scheme.
The Hilbert functor $\underline{{\Hilb}_{h}^{G}(W)}$: ${Sch}^{op} \to \Sets$ is defined by 
$$S \mapsto \left\{ 
\begin{array}{c} 
\xymatrix{
\ZZ \ar@{.>}_{\pi}[dr]  \ar@{^{(}->}[r] &   S \times W \ar@{->}[d]^{pr_1}\\
                     &  S } \end{array}
\middle| 
\begin{array}{l} \ZZ \ \text{is a $G$-stable closed subscheme;}  \\
{\pi}\ \text{is a flat morphism; and}\\
 {\AA}_{(M)} \ \text{is locally free of rank $h(M)$ over } S  \end{array} 
\right\} .$$ 

An element $(\pi: \ZZ \to S) \in \underline{{\Hilb}_{h}^{G}(W)}(S)$ is called a \textit{flat family over $S$ of $G$-stable closed subschemes of $W$}.
\end{definition} 

By \cite[Theorem 2.11]{Br}, the Hilbert functor $\underline{{Hilb}_{h}^{G}(W)}$ is represented by a quasi-projective scheme $\HH={\Hilb}_{h}^{{G}}(W)$: the invariant Hilbert scheme associated with the affine $G$-scheme $W$ and the Hilbert function $h$. 
We denote by $\UU \subset W \times \HH$ the \textit{universal subscheme}, and we write $\pi:\UU\to\HH$ for the  \textit{universal family}, that is, the projection to the second factor. 

Let $\nu:W \to W/\!/G =\Spec(k[W]^G)$ be the quotient morphism, which is induced by the inclusion of algebras $k[W]^G \subset k[W]$. If the affine scheme $W$ is reduced and $W/\!/G$ is irreducible, then by \cite[Theorem 14.4]{Ei} the quotient morphism is flat over a non-empty open subset $U \subset W/\!/G$; the Hilbert function of the family $\nu_{| \nu^{-1}(U)}: \nu^{-1}(U) \to U$ is called \textit{Hilbert function of the general fibers of $\nu$} and is denoted by $h_0$. The next proposition will be useful in sections \ref{HclassSOngeneral} and \ref{two others}.

\begin{proposition} \emph{(\cite[Proposition 3.15]{Br})} \label{chow}
Let $W$ be a reduced, affine $G$-scheme such that the quotient $W/\!/G$ is irreducible, let $h_0$ be the Hilbert function of the general fibers of $\nu$, and let $\HH=\Hilb_{h_0}^G(W)$ be the corresponding Hilbert scheme.
Then there exists a projective morphism $\gamma:\HH \to W/\!/G$, called the \emph{Hilbert-Chow morphism}, such that the diagram
\begin{equation} \label{diag1}
        \xymatrix{
     \UU \ar[d]_(0.4){\pi} \ar[r]^(0.5){q} & W \ar@{->>}[d]^(0.4){\nu} \\
     \HH \ar[r]_(0.4){\gamma} & W/\!/G 
    }
\end{equation}
commutes, where $\UU \subset W \times \HH$ is the universal subscheme, and the morphisms from $\UU$ are the projections. 
Moreover, $\gamma$ is an isomorphism over the the biggest open subset $U\subset W/\!/G$  over which $\nu$ is flat.
\end{proposition}

\begin{definition}
In the setting of Proposition \ref{chow}, we define the main component $\HHm$ of $\HH$ as the Zariski closure of $\gamma^{-1}(U)$. 
It is an irreducible component of $\HH$ which is mapped birationally to the quotient $W/\!/G$ by the Hilbert-Chow morphism.
\end{definition}

\subsection{Fixed points for the action of a Borel subgroup} \label{subsec fix borel}

Let us fix an algebraic subgroup $H$ of the $G$-equivariant automorphism group 
$${\Aut}^G(W):=\{h \in \Aut(W) \mid \forall g \in G,\ h \circ g=g \circ h \}.$$ 
Then we have the following:  

\begin{proposition} \emph{(\cite[Proposition 3.10]{Br})} \label{groupaction}
With the above notation, the group $H$ acts on the invariant Hilbert scheme $\HH$ and on the universal subscheme $\UU$, and the universal family $\UU \to \HH$ is $H$-equivariant. Moreover, in the setting of Proposition \ref{chow}, the Hilbert-Chow morphism $\gamma:\HH \to W/\!/G$ is also $H$-equivariant.
\end{proposition}

Let us fix a Borel subgroup $B \subset H$. Recall that our strategy to determine the global structure of $\HH$, e.g. connected components, irreducible components, their dimension and singular locus etc, is based on the study of the fixed points of $\HH$ for the $B$-action, and was explained at the end of the introduction. The following result justifies the importance of these fixed points: 

\begin{lemma}  \label{closedpoints}
Let us assume that $W/\!/G$ has a unique closed $H$-orbit, and that this orbit is a point $\{o\}$.
Then each $H$-stable closed subset of $\HH$ contains a fixed point for the action of the Borel subgroup $B$. In particular, there is a one-to-one correspondence between closed $H$-orbits in $\HH$ and $B$-fixed points in $\HH$.
\end{lemma}

\begin{proof}
The Hilbert-Chow morphism $\gamma:\HH \to W/\!/G$ is projective (Proposition \ref{chow}) and $H$-equivariant (Proposition \ref{groupaction}), hence $\gamma^{-1}(o)$ is a projective $H$-scheme. If $C$ is any $H$-stable closed subset of $\HH$, then $\gamma(C)$ is a $H$-stable closed subset of $W/\!/G$, and thus $o \in \gamma(C)$, that is, $F:=C \cap \gamma^{-1}(o)$ is non-empty. Hence, $F$ contains at least one $B$-fixed point by Borel's fixed-point theorem (\cite[Theorem 10.4]{Bo}). The last statement of the lemma follows from the fact that for every parabolic subgroup $Q \subset G$ containing $B$ the flag variety $G/Q$ has a unique $B$-fixed point.
\end{proof}

If $X \subset W$ is a $G$-stable closed subscheme, then we denote by $[X] \in \HH$ the corresponding closed point, and by $I \subset P:=k[W]$ the ideal of $X$. 

\begin{proposition} \emph{(\cite[Proposition 3.5]{Br})} \label{isoTangent}
With the notation above, let $T_{[X]} \HH$ be the tangent space of $\HH$ at $[X]$. 
Then there is a canonical isomorphism
$$T_{[X]} \HH \cong N^G_{X/W}:=\Hom_{P}^{G}(I,P/I),$$
where $\Hom_{P}^{G}$ stands for the space of $P$-linear, $G$-equivariant maps.
\end{proposition}

\begin{remark}
If $H$ is an algebraic subgroup of ${\Aut}^G(W)$, and if $[X] \in \HH$ is a $H$-fixed point, then the tangent space $T_{[X]} \HH$ is an $H$-module and the isomorphism of Proposition \ref{isoTangent} is $H$-equivariant.   
\end{remark}

\subsection{The Reynolds operator}

Let $G$ be a reductive group, and let $V$ be a rational $G$-module. Then there is a unique $G$-stable complement $V'$ to $V^G\subset V$. According to the decomposition $V=V^G\oplus V'$, we write an element $v\in V$ as $v=v_0 + v'$, where $v_0 \in V^G$ and $v' \in V'$.

\begin{definition}
The projection on the $G$-invariant part 
\[
\RR_V:V\to V^G, \quad v\mapsto v_0
\]
is called the \emph{Reynolds operator} of the $G$-module $V$.
\end{definition}

The Reynolds operator has the following useful property:

\begin{proposition}
Let $V_1, V_2$, and $V$ be rational $G$-modules. Let $\mu :V_1\otimes V_2 \to V$ be a $G$-equivariant morphism. Then
\begin{enumerate}
\item\label{prop enum reynolds} $\RR_V(\mu(v_1 \otimes v_2)) = \mu(\RR_{V_1}(v_1)\otimes v_2)$ if $v_2\in V_2^G$; and
\item $\RR_V(\mu(v_1 \otimes v_2)) = \mu(v_1\otimes \RR_{V_2}(v_2))$ if $v_1\in V_1^G$.
\end{enumerate}
\end{proposition}

\begin{proof}
For symmetry reasons it suffices to show \eqref{prop enum reynolds}. As $\mu$ clearly sends $V_1^G\otimes V_2^G$ to $V^G$ it remains to show that the restriction of $\RR_V\circ\mu$ to $V'_1\otimes V_2^G$, where $V'_1$ is the $G$-stable complement to $V_1^G$ in $V_1$, is the zero map. But this is straightforward as $V^G$ is the trivial $G$-module, $V'_1\otimes V_2^G$ consists only of non-trivial $G$-modules, and $\RR_V\circ\mu$ is a $G$-equivariant morphism.
\end{proof}

Let $V_1$ and $V_2$ be two $G$-modules. Consider the induced action on $\Hom(V_1,V_2)$ given by $g\cdot F:= g\circ F \circ g^{-1}$. A linear map $F:V_1\to V_2$ is invariant for this action if and only if it is a $G$-equivariant linear map. The next corollary will be very useful in the proof of Proposition \ref{lemma ext1 obstruktionsraum}, and also in section \ref{section algorithm}.

\begin{corollary}\label{cor reynolds}
Let $V_1, V_2$, and $V_3$ be rational $G$-modules. Let $F_1:V_1 \to V_2$, and $F_2:V_2\to V_3$ be morphisms. Then
\begin{enumerate}
\item $\RR (F_2\circ F_1) = \RR(F_2)\circ F_1$ if $F_1$ is $G$-equivariant; and
\item $\RR (F_2\circ F_1) = F_2\circ \RR(F_1)$ if $F_2$ is $G$-equivariant.
\end{enumerate}
\end{corollary}

To calculate the Reynolds operator in our algorithm of section \ref{section algorithm}, we implemented Algorithm 4.5.19 of \cite{DK} using the computer algebra system \cite[Macaulay2]{Mac2}. This algorithm uses the action of the Lie algebra of $G$ and more particularly the Casimir operator on $k[W]$.

\section{Invariant deformation theory}\label{sec defo}

This section is the heart of this article. Given a set of equations defining a $G$-stable closed subscheme $X\subset W$ corresponding to a point of some invariant Hilbert scheme $\HH=\Hilb_h^G(W)$, our goal is to obtain a presentation by generators and relations of the completed local ring $\OH$.

In section \ref{sssection2} we introduce the main objects of $G$-invariant deformation theory such as the deformation functor and the universal  deformation, and we also make the link with the invariant Hilbert scheme. The remainder of section \ref{sssection2} is dedicated to the explicit computation of these objects. In section \ref{subsec tangent spaces}, we describe the starting point and the general strategy to determine a presentation of the ring $\OH$. In section \ref{subsec obstruction}, we recall the definition of an obstruction theory and of an obstruction map. Section \ref{G eq reso} adapts some well-known technical results on $G$-equivariant presentations to our framework. This is used in section \ref{subsec obstruction 2}, where we describe in detail an obstruction space and an obstruction map for our deformation functor. Our main result, Theorem \ref{lemma cle}, shows how to ``explicitly compute'' the ring $\OH$ modulo an arbitrary power of its maximal ideal, building on an explicit description of obstruction theory which we present in section \ref{subsec obstruction explicit}.

\subsection{Setup}  \label{subsec setup}

Let $G$ be a reductive algebraic group, and let $W$ be an affine $G$-scheme of finite type. We fix a $G$-stable closed subscheme $X\subset W$, whose coordinate ring is denoted by $k[X]$, such that 
\begin{equation} \label{mod covariants}
\forall M \in \Irr(G), \  \dim(\Hom^G(M,k[X])) < \infty.
\end{equation}
It follows from \cite[Lemma 2.1]{Br10} that the condition \eqref{mod covariants} is in fact equivalent to $\dim_k(k[X]^G)<\infty$, which means that $X$ has a finite number of closed $G$-orbits. We call 
$$h_X:\Irr(G) \to \NN,\ M \mapsto \dim(\Hom^G(M,k[X]))$$ 
the \textit{Hilbert function of $X$}, and we denote by 
$$ \HH=\Hilb_{h_X}^G(W)$$
the invariant Hilbert scheme associated with the affine $G$-scheme $W$ and the Hilbert function $h_X$. The subscheme $X\subset W$ defines a point $[X] \in \HH$, and we denote by $\OH$ the formal completion of the local ring $\OO_{\HH,[X]}$.

\subsection{The deformation functor}  \label{sssection2}

We recall the formalism of deformation functors; see \cite{Sern} for an introduction. For this, we need some notation:

\begin{itemize} \renewcommand{\labelitemi}{$\bullet$}
\item $\Art :=\{\text{Artinian local } k\text{-algebras with residue field } k\};$ and
\item $\Noeth:=\{\text{complete Noetherian local } k\text{-algebras with residue field } k\}$.
\end{itemize}

Note that $\Art$ is a full subcategory of $\Noeth$. For $A \in \Noeth$, we denote by $\mm_A$ its maximal ideal. Every $A\in \Noeth$ is the inverse limit of the $A/\mm_A^n \in \Art$.

\begin{definition}  \label{functorDef}
We define the functor $D: \Art \to \Sets$ of infinitesimal $G$-stable deformations of $X$ inside $W$ by:
$$\begin{array}{cccl}
  A & \mapsto & \left\{ 
\begin{array}{c} 
\xymatrix{
\ZZ \ar@{.>}_{\pi}[dr]  \ar@{^{(}->}[r] &   \Spec(A) \times W \ar@{->}[d]^{pr_1}\\
                     &  \Spec(A) } \end{array}
\middle| 
\begin{array}{l} \ZZ  \text{ is a $G$-stable closed subscheme;}  \\ 
{\pi} \text{ is a flat morphism; and}\\
 {\pi}^{-1}(y_A)=X.
 \end{array} 
\right\}, 
\end{array}$$
where $G$ acts trivially on $A$, and $y_A$ is the subscheme defined by the unique maximal ideal of $A$.
\end{definition}

A covariant functor $F: \Art \to \Sets$ is called \emph{prorepresentable} if there exists $\Ach \in \Noeth$ and an  isomorphism of functors $\phi:\Hom_{\Noeth}(\Ach,.) \to F$, where we restrict $\Hom_{\Noeth}(\Ach,.)$ to $\Art\subset \Noeth$. If such a couple $(\Ach,\phi)$ exists, it is unique up to unique isomorphism (\cite[Exercise 15.1]{Ha2}). Such $\Ach$ is said to \emph{prorepresent} the functor $F$.

The functor $D$ defined above was considered for the first time by Cupit-Foutou in \cite[\S 3.4]{CF}; in particular, she mentions the next result, which is an easy consequence of the universal property of the invariant Hilbert scheme.

\begin{lemma}  \label{repD}
The functor $D$ from Definition \ref{functorDef} is  prorepresented by $\OH$. 
\end{lemma}

\begin{proof}
For $A \in \Art$, we have 
$$\underline{{\Hilb}_{h_X}^{G}(W)}(\Spec(A)) \cong \Hom_{Sch}(\Spec(A),\HH),$$
and thus
\begin{align*}
D(A) &\cong \{ \phi \in \Hom_{Sch}(\Spec(A),\HH) \ |\ \phi(y_A)=[X]\} \\
     &\cong \{ \phi \in \Hom_{Sch}(\Spec(A),\Spec(B)) \ |\ \phi(y_A)=[X]\}, \\
     & \text{ where $\Spec(B)$ is an open affine neighborhood of  $[X]$ in $\HH$,}\\  
     &\cong \{ \psi \in \Hom_{k\text{-alg}}(B,A) \ |\ \psi^{-1}(\mm_A)=\mm_X\},\\
     & \text{ where $\mm_X$ is the maximal ideal of $B$ corresponding to $X$,}\\
     &\cong \{ \psi \in \Hom_{k\text{-alg}}(B_{\mm_X},A) \ |\ \psi^{-1}(\mm_A)=\mm_X\},\\
     & \text{ where $B_{\mm_X}(=\OO_{\HH,[X]})$ is the localization of $B$ at $\mm_X$,}\\
     &\cong \Hom_{\Noeth}(\OH,A), 
\end{align*}
where the last isomorphism is a consequence of the fact that $A$ is complete. 
We conclude that $\OH$ prorepresents the functor $D$. 
\end{proof}

To simplify the notation, we put 
\begin{equation}\label{eq def R}
\Rch:=\OH.
\end{equation}
Let us mention that, at the end of the next section, we will define a polynomial ring $R$ whose completion along the irrelevant maximal ideal is $\Rch$. 
We fix once and for all an isomorphism of functors
\begin{equation}  \label{isoPhi}
\Hom_{\Noeth}(\Rch,.) \to D.
\end{equation}

The natural morphisms 
$$\Spec(\Rch) \to  \Spec(\OO_{\HH,[X]}) \to \HH$$
and the universal family $\UU \to \HH$ induce the $G$-invariant morphism
$$\XX:=\UU \times_{\HH} \Spec(\Rch) \to \Spec(\Rch).$$
The latter is called the \textit{universal $G$-stable deformation of $X$ inside $W$}. 
When there is no danger of confusion, we will just speak of the \emph{universal deformation}.
We will refer to $\Spec(\Rch)$ or $\Rch$ as the \textit{base space} of the universal deformation.

\subsection{Tangent spaces and algorithmic problem}\label{subsec tangent spaces}

Consider
$$T^1 :=D(k[t]/(t^2)),$$
which is called the \emph{tangent space to the deformation functor}. By means of the isomorphism \eqref{isoPhi}, the set $T^1$ is endowed with a vector space structure, namely $T^1 \cong (\mm_{\Rch}/\mm_{\Rch}^2)^\vee$, and thus $T^1$ is nothing else than the tangent space to the invariant Hilbert scheme at the point $[X]\in \HH$. We will also refer to $T^1$ as the space of \emph{first order $G$-stable deformations of $X$ inside $W$}.

Denote by 
$$M:=(T^1)^\vee \cong \mm_{\Rch}/\mm_{\Rch}^2$$ 
the dual space of $T^1$, and by 
\begin{equation} \label{def S}
S:=\Sym^\bullet M
\end{equation}
the polynomial algebra generated by $M$. Next, define $\mm_S$ to be the maximal ideal of $S$ generated by $M$, and $\Shut$ to be the $\mm_S$-adic completion of $S$. 

Now let $d:=\dim( T^1)$, and choose once and for all elements 
\begin{equation} \label{eq choice of yi}
y_1, \ldots, y_d \in \mm_{\Rch}
\end{equation}
such that their images $t_1, \ldots, t_d \in \mm_{\Rch}/\mm_{\Rch}^{2} \cong M$ form a basis. The proof of the next lemma is elementary. 

\begin{lemma} \label{eq choice}
The morphism $\psi:\Shut \to \Rch,\ t_i \mapsto y_i$ is surjective. \qed
\end{lemma}

In fact, it follows from the following -- more general but equally elementary -- statement that will be used several times later on.

\begin{lemma} \label{lemma fact}
Any morphism $A'\to A$ in $\Noeth$ which is surjective on the Zariski cotangent spaces is itself surjective. Any endomorphism $A\to A$ in $\Noeth$ which is surjective on the Zariski cotangent spaces is an automorphism. \qed
\end{lemma}

According to Lemma \ref{eq choice}, calculating $\Rch$ is tantamount to calculating the ideal
\begin{equation}\label{eq def khut}
\hat K:=\Ker(\psi) \subset \mm_{\Shut}^2.
\end{equation}
We will do this step by step in section \ref{section algorithm}, that is, we will device an algorithm calculating $\hat K + \mm_\Shut^{n+1}$ for each $n \geq 1$. Notice also that the ideal $\hat K + \mm_\Shut^{n+1}$ is generated by polynomials in $S$. Therefore, we may equally well perform all our calculations in $S$. The goal is then to calculate 
\begin{equation}\label{eq def kn}
K_n:= \left(\hat K + \mm_\Shut^{n+1}\right) \cap S \subset \mm_S^2.
\end{equation}

Notice that $S/K_n  \cong \Rch/\mm_{\Rch}^{n+1} \cong R/\mm_R^{n+1}$, where we denote by $R\subset \Rch$ the ring of polynomials in the $y_i$; in other words, $R$ is the image of $S$ under the morphism $\psi$ of Lemma \ref{eq choice}.

\subsection{Obstruction Spaces}\label{subsec obstruction}

An important ingredient for calculating the universal deformation step by step is Schlessinger's notion of a \emph{small extension}. This is an exact sequence
\begin{equation}\label{eq small extension}
 0 \to J \to A' \to A \to 0,
\end{equation} 
where $A,\ A' \in \Art$, and $J$ is an ideal of $A'$ satisfying $\mm_{A'}J=0$. Thus, the $A'$-module structure of $J$ factors through $k=A'/\mm_{A'}$, and $J$ is nothing more than a vector space (over $k$).

\begin{definition}   \label{def obstruction theory}
An \textit{obstruction theory} for a covariant functor $F: \Art \to \Sets$ is the following datum:
\begin{itemize} 
\item a finite dimensional vector space $V_F$; and
\item for each small extension as in \eqref{eq small extension}, a map $$\ob:F(A) \to J \otimes V_F,$$ 
\end{itemize}

with the following properties:

\begin{enumerate} 
\item The sequence $$ F(A') \to F(A) \stackrel{ob}{\to} J \otimes V_F$$
is  exact.

\item A commutative diagram between small extensions 
$$ \xymatrix{
 0 \ar[r] &J_1 \ar[r] \ar[d]_{f} & A' \ar[r] \ar[d] & A \ar[r] \ar[d] &         0 \\
 0 \ar[r] &J_2 \ar[r] & B \ar[r]  & B' \ar[r]  &         0 
 }$$
gives rise to the commutative diagram with exact rows
 $$\xymatrix{
  F(A') \ar[r] \ar[d] &F(A) \ar[r]^(0.4){ob_1} \ar[d] &  J_1 \otimes V_F   \ar[d]^{f \otimes \id}   \\
  F(B) \ar[r]        &F(B') \ar[r]^(0.4){ob_2}        &          J_2 \otimes V_F
 }$$
\end{enumerate} 
The map $\ob$ is called an \emph{obstruction map}, and the vector space $V_F$ is called an \emph{obstruction space}.
\end{definition}

For a given functor with an obstruction theory, the obstruction spaces are by no means unique. One could for example use any space containing a given obstruction space $V_F$.

The argument of \cite[Example 11.0.2]{Ha2} guarantees that the deformation functor $D$ of Definition \ref{functorDef} has an obstruction theory, an obstruction space being given by $V_D:=(\hat K/\mm_{\Shut}\hat K)^{\vee}$. However, this is  rather an abstract existence result, and $V_D$ is not directly accessible to calculations as the whole story is about calculating $\hat K$. 

In most practical situations there are canonical obstruction spaces. We will exhibit one for the functor $D$ in section \ref{subsec obstruction 2}, but first we need a digression on $G$-equivariant presentations.

\subsection{\texorpdfstring{$G$}{G}-equivariant presentations}  \label{G eq reso}

We take the same notation as in section \ref{subsec setup}, and we abbreviate $P:=k[W]$ and $P_A:=P\otimes A$ for a $k$-algebra $A$. Let $I \subset P$ be the defining ideal of the $G$-stable closed subscheme $X \subset W$. We choose once and for all a $G$-equivariant presentation of $P/I$ (one easily checks that such a presentation always exists), that is, we take $N_1 \subset P$ and $N_2 \subset P \otimes N_1$ two $G$-submodules, of dimension $n_1$ and $n_2$ respectively, such that there is an exact sequence of $(P, G)$-modules:
\begin{equation}\label{eq equivariant resolution}
 \begin{array}{cccccccccc}
(\Pi_0) & &P \otimes N_2 & \stackrel{v_0}{\to} &P \otimes N_1 & \stackrel{u_0}{\to} &P & \to &P/I & \to 0  \\
&&1 \otimes r & \longmapsto & r & & && & \\
&&&& 1 \otimes f & \longmapsto & f &&&
\end{array}
\end{equation}
Hereby, $G$ acts on $P, N_1$, and $N_2$, while the $P$-module structure is induced by that of $P$ given by the multiplication. Notice that $u_0$ and $v_0$ are morphisms of $(P,G)$-modules; in particular, they are $G$-equivariant.

The next statement is a slight generalization of \cite[Theorem A.10]{Sern}. The proof is similar to the one given by Sernesi but we chose to give a sketch since, in our case, a reductive group $G$ acts on $P$, and all morphisms have to be $G$-equivariant.

\begin{theorem} \label{eqMat}
With the above notation, let $A'\to A$ be a surjection in $\Art$. Let $\pi:\ZZ \to \Spec (A)$ be a family of $G$-stable flat subschemes of $W$, and denote by $I_\ZZ \subset P_A$ the defining ideal of $\ZZ$. Then there exists an exact sequence
\begin{equation}\label{eq resolution pi}
(\Pi) \hspace{3mm} 
\begin{array}{cccccccc}
P_{A} \otimes N_2 & \stackrel{v}{\to} & P_{A} \otimes N_1 & \stackrel{u}{\to} & P_{A} & \to &P_A/I_\ZZ & \to 0.  \\
\end{array}
\end{equation}
of $(P_A,G)$-modules and, given a sequence $(\Pi_0)$ as in \eqref{eq equivariant resolution}, the sequence $(\Pi)$ may be chosen to satisfy $(\Pi)\otimes k = (\Pi_0)$.
Moreover, the following are equivalent for an ideal $I_{\ZZ'}\subset P_{A'}$:
\begin{enumerate}
\item \label{ffllaatt} $\ZZ':=\Spec \left(P_{A'}/I_{\ZZ'}\right)$ is $G$-stable and flat over $A'$, and $(P_{A'}/I_{\ZZ'})\otimes_{A'}A = P_{A}/I_{\ZZ}$;
\item\label{thm item exact sequence} there exists an exact sequence of $(P_{A'},G)$-modules
\begin{equation*}
(\Pi') \hspace{3mm} \begin{array}{cccccccc}
P_{A'} \otimes N_2 & \stackrel{v'}{\to} & P_{A'} \otimes N_1 & \stackrel{u'}{\to} & P_{A'} & \to &P_{A'}/I_{\ZZ'} & \to 0  \\
\end{array}
\end{equation*}
such that $(\Pi')\otimes_{A'} A=(\Pi)$; and
\item\label{thm enum lifting relations} there exists a complex as in \eqref{thm item exact sequence} which is exact except possibly at $P_{A'} \otimes N_1$.
\end{enumerate}
\end{theorem}

\begin{proof}
Let us show the existence of the presentation $(\Pi)$ such that $(\Pi)\otimes k = (\Pi_0)$ directly. Choose a morphism of $P_A$-modules $u:P_A \otimes N_1 \to I_{\ZZ}$ such that the following diagram commutes
$$ \xymatrix{
    P_A \otimes N_1 \ar[r]^{u} \ar[d]_{p_1}  & I_\ZZ \ar[d]^{p_2} \\
    P \otimes N_1 \ar[r]_{u_0} & I
  }$$
where the $p_i$ are the morphisms induced by $A \to A/\mm_A=k$. As $u_0$, $p_1$, and $p_2$ are $G$-equivariant, $u$ can be chosen $G$-equivariant. As $P_A/I_{\ZZ}$ is a flat $A$-module, we obtain $\Tor_1^A(P_A/I_\ZZ,k)=0$. Tensoring the exact sequence
$$0 \to I_\ZZ  \to P_A  \to P_A/I_\ZZ \to 0 $$ 
over $A$ with $k$ and using $(P_A/I_\ZZ) \otimes_A k = P/I$, we see that $I_\ZZ \otimes_A k=I$. It follows from Nakayama's lemma that $u$ is surjective. Arguing as before, we find a surjective morphism of $(P_A,G)$-modules $v:P_A \otimes N_2 \to \Ker(u)$ such that the following diagram commutes
$$ \xymatrix{
    P_A \otimes N_2 \ar[r]^{v} \ar[d]_{p_1} & \Ker(u) \ar[d]^{p_2} \\
    P \otimes N_2 \ar[r]_{v_0} & \Ker(u_0)
  }$$ 
The proof of equivalence of the conditions $(1)$, $(2)$, and $(3)$ requires similar arguments and is left to the reader.
\end{proof}

\begin{remark}
We say that the couple $(u,v)$ \textit{represents} the family $\ZZ \to \Spec(A)$; this couple is not unique in general. 
The implication \eqref{thm enum lifting relations}  $\Rightarrow$ \eqref{ffllaatt} tells us that, to construct a flat family, we only have to lift relations and we do not need to care about verifying exactness properties.
\end{remark}

Applying the left exact contravariant functor $\Hom_P(\cdot,P/I)$ to the presentation \eqref{eq equivariant resolution} and taking the $G$-invariants (which is a right exact functor since $G$ is reductive), we get the exact sequence of vector spaces
\begin{equation*} 
  \xymatrix{
    0 \ar[r] &{\Hom}_{P}^{G}(I,P/I) \ar[r]^(0.45){u_0^*} & {\Hom}_{P}^{G}(P {\otimes} N_1,P/I) \ar[r]^{v_0^*} \ar[d]^{\cong} & {\Hom}_{P}^{G}(P \otimes N_2,P/I) \ar[d]^{\cong} \\
        &   &  \Hom^G(N_1,P/I) & \Hom^G(N_2,P/I)
  }
\end{equation*}
Together with Proposition \ref{isoTangent}, this sequence implies that 
$$T_{[X]} \HH  \isom \Ker\left(v_0^*: \Hom^G(N_1,P/I) \to \Hom^G(N_2,P/I)\right),$$
and thus $\dim(T_{[X]} \HH)=\dim(\Hom^G(N_1,P/I))-\rg(v_0^*)$. This observation enables us to compute the tangent space $T_{[X]} \HH$ algorithmically; see  section \ref{first order}.

\subsection{Obstruction spaces II}\label{subsec obstruction 2}

In the setting of section \ref{subsec setup}, there is a canonical obstruction space for the deformation functor $D$ of Definition \ref{functorDef}.

\begin{proposition} \label{lemma ext1 obstruktionsraum}
An obstruction space for the deformation functor $D$ is given by $\Ext^{1,G}_P(I,P/I)$, where $P=k[W]$ and $I \subset P$ is the ideal of $X$.
\end{proposition}

\begin{proof}
Let $A'\to A$ be a small extension in $\Art$, and let $J:=\Ker(A'\to A)$. Given some $\lambda \in D(A)$ we have to construct an element $\ob(\lambda) \in \Ext^{1,G}_P(I,P/I) \otimes J$ which is $0$ if and only if there exists $\lambda'\in D(A')$ restricting to $\lambda$ over $A$. 

Let $I_\ZZ \subset P_A=P \otimes A$ be the ideal of the $G$-stable deformation of $X$ over $A$ given by $\lambda$. Suppose there were an ideal $I_{\ZZ'}\subset P_{A'}$ such that $P_{A'}/I_{\ZZ'}$ is flat over $A'$, and $(P_{A'}/I_{\ZZ'}) \otimes_{A'} A=P_A/I_{\ZZ}$. Then the exact sequence
$$\Tor_{A'}^{1}(P_{A'}/I_{\ZZ'},A) \to I_{\ZZ'} \otimes_{A'} A \to P_{A'} \otimes_{A'} A \to (P_{A'}/I_{\ZZ'}) \otimes_{A'} A $$ 
implies that $I_{\ZZ'} \otimes_{A'} A=I_{\ZZ}$, since $\Tor_{A'}^{1}(P_{A'}/I_{\ZZ'},A)=0$, and that $I_{\ZZ'}$ is flat over $A'$. Observe that the kernel of the restriction $I_{\ZZ'}\to I_\ZZ$ is canonically isomorphic to $I\otimes J$. This follows from flatness and the fact that the multiplication $I_{\ZZ'}\otimes J\to I_{\ZZ'}$ factors through $I_{\ZZ'}\otimes J \to I\otimes J$ as $\mm'.J =0$, where $\mm'$ denotes the maximal ideal of $A'$. So extensions of $I_\ZZ$ to $A'$ are extensions of $I_\ZZ$ by $I\otimes J$. 

We want to construct an extension of $I_\ZZ$ to $A'$ by using the description of Theorem \ref{eqMat}. Fix a $G$-equivariant presentation $(\Pi)$ of $k[\ZZ]$ as in Theorem \ref{eqMat}, and denote by $u$ and $v$ the corresponding morphisms. One may check that, similarly as for $I_{\ZZ'}$, a presentation $(\Pi')$ of $I_{\ZZ'}$ is an extension of $(\Pi)$ by $(\Pi_0)\otimes J$.
Hence, in order to obtain a flat extension of $I_\ZZ$ to $A'$, we have to find $G$-equivariant morphisms $u'$ and $v'$ completing the following diagram
\begin{equation}\label{eq construction of obstruction}
\xymatrix{
P \otimes N_2 \otimes J\ar[r]^{v_0\otimes\id}\ar[d] &P\otimes N_1 \otimes J \ar[r]^(0.6){u_0\otimes\id}\ar[d] &P\otimes J \ar[d]^{q}\\
P_{A'}\otimes N_2 \ar@{-->}[r]^{v'} \ar[d]&P_{A'}\otimes N_1 \ar@{-->}[r]^{u'} \ar[d]&P_{A'}\ar[d] \\
P_A\otimes N_2 \ar[r]^v &P_A\otimes N_1 \ar[r]^u &P_A
}
\end{equation}
such that each square commutes, and $u'\circ v' =0$ (flatness will follow from Theorem~\ref{eqMat}). Here $u_0 = u\otimes_A k$ and $v_0 = v\otimes_A k$ by assumption.

As the $P_{A'}$-modules in the middle row are free, we may find horizontal arrows making the lower squares commute. The commutativity of the upper squares is automatic and does not depend on the choice of $u'$ and $v'$ as long as the lower square commutes. This again is a consequence of $\mm'\cdot J =0$. Replacing $(u',v')$ by $(\RR(u'), \RR(v'))$, where $\RR$ denotes the Reynolds operator, we may assume that all morphisms are $G$-equivariant. Indeed, the application of the Reynolds operator does not affect the commutativity of Diagram \eqref{eq construction of obstruction} by Corollary \ref{cor reynolds}. 

Thus, for a given $\lambda \in D(A)$, we can always construct a commutative diagram like Diagram \eqref{eq construction of obstruction}, but $\lambda$ is the restriction of some $\lambda' \in D(A')$ if and only if the morphisms $(u',v')$ can be chosen such that $u' \circ v'=0$. The commutativity of Diagram \eqref{eq construction of obstruction} implies that the composition $u' \circ v'$ takes values in $P\otimes J$, and thus there exists a map $\alpha:P_{A'} \otimes N_2 \to P \otimes J$ such that $q \circ \alpha=u' \circ v'$. The composition of the map $\alpha$ with the projection $\psi:P\otimes J \to (P/I)\otimes J$ gives an element 
\begin{align*}
\eta \in \Hom_P^G(P_{A'}\otimes N_2, (P/I)\otimes J) &\cong \Hom_P^G(P\otimes N_2, (P/I)\otimes J) \\
                     &\cong \Hom_P^G(P\otimes N_2, P/I) \otimes J 
\end{align*}
where first isomorphism follows from $\mm'.J=0$, and the second isomorphism follows from the fact that $G$ acts trivially on $J$.
 
If we enlarge the lines of Diagram \eqref{eq construction of obstruction} one step further to the left and consider the sequence
\begin{equation}\label{eq ext0}
\Hom^{G}_P(P \otimes N_1,P/I)  \to[d_1] \Hom^{G}_P(P \otimes N_2,P/I)  \to[d_2] \Hom^{G}_P(P \otimes N_3,P/I),  
\end{equation}
then we see that $\eta$ is in fact a cocycle and determines an element 
$$\overline{\eta} \in \Ext^{1,G}_P(I,P/I) \otimes J=(\Ker(d_2)/\Im(d_1)) \otimes J.$$ 
We claim that it is possible to change $u'$ and $v'$ such that their product is zero if and only if $\overline{\eta}=0$, and we put
\begin{equation} \label{ob map}
\ob(\lambda):=\overline{\eta}.
\end{equation}
   
To verify this last claim suppose that $\eta$ is a boundary. Then there is some $G$-equivariant $\xi:P_{A'} \otimes N_1  \to (P/I)\otimes J$ such that $\eta = \xi \circ v'$. We lift $\xi$ to a $G$-equivariant morphism $\tilde \xi :P_{A'} \otimes N_1 \to P \otimes J$. Note that we may always first take an arbitrary lift and then apply the Reynolds operator. Then we replace $u'$ by $u'':=u' - \tilde \xi$. The composition $u''\circ v'$ takes values in $\Ker( \psi) = I\otimes J \subset P\otimes J$, and hence we find an equivariant map $\delta : P_{A'}\otimes N_2 \to P\otimes N_1 \otimes J$ such that $q\circ (u_0\otimes \id) \circ \delta = u''\circ v'$. Replacing $v'$ by $v'':= v' - \delta$ we thus have $u''\circ v'' = 0$ and obtain a flat $G$-stable extension $I_{\ZZ'}:=\Im( u'')$ of $I_\ZZ$ to $A'$.

For the converse we have to read the preceding paragraph backwards. Suppose there were $u''$ and $v''$ fitting in the diagram \eqref{eq construction of obstruction}. Then we have to show that the $\overline \eta$ defined from $u'\circ v'$ is a coboundary. We obtain $\delta$ as $v'-v''$ as above. It suffices to show that $\eta - (u_0\otimes \id)\circ \delta$ is a coboundary. But this is equal to $(u'-u'')\circ v = \tilde \xi \circ v$ and $\xi=\psi\circ\tilde \xi$ satisfies $\eta = \xi \circ v'$.

Finally, condition (2) of Definition \ref{def obstruction theory} is tedious but straightforward.
\end{proof}

The obstruction space given by Proposition \ref{lemma ext1 obstruktionsraum} is quite reasonable, but still not optimal for our purposes. We will introduce a more convenient obstruction space in Corollary \ref{lemma obstruction inclusion}.

\begin{notation} \label{not avec base}
For $i=1,2$ we identify $P \otimes N_i$ with $P^{n_i}$. Such identifications are equivalent to fixing bases of the vector spaces $N_1$ and $N_2$. Note that $G$ acts not only on the coefficients of $P^{n_i}$ but also on the basis vectors.  
\end{notation}
 
Let $(U_0,V_0) \in \Hom_P(P^{n_1},P) \times \Hom_P(P^{n_2},P^{n_1})$ be matrices representing the morphisms $(u_0,v_0)$ of the exact sequence \eqref{eq equivariant resolution}. By definition of $u_0$ and $v_0$, the matrices $U_0$ and $V_0$ are $G$-equivariant. Consider the morphism of $(P,G)$-modules
\begin{small}
\begin{equation}\label{eq obstruktionsraum}
\mu:\Hom_P(P^{n_1},P) \oplus \Hom_P(P^{n_2},P^{n_1}) \to \Hom_P(P^{n_2},P), \quad (C,D)\mapsto C V_0 + U_0 D,
\end{equation}
\end{small}
and define
\begin{equation}\label{eq def nsansg}
N:=\Coker(\mu).
\end{equation}
At this point, it is useful to deal also with the non-$G$-equivariant situation, as the standard procedure to solve equations for $G$-equivariant matrices is to solve the equation for arbitrary matrices, and then to apply the Reynolds operator. Define also
\begin{small}
\begin{equation}\label{eq obstruktionsraum equiv}
\mu^G: \Hom_P^G(P^{n_1},P) \oplus \Hom_P^G(P^{n_2},P^{n_1}) \to \Hom_P^G(P^{n_2},P), \quad (C,D)\mapsto C V_0 + U_0 D
\end{equation}
\end{small}
and
\begin{equation}\label{eq def navecg}
N^G:=\Coker(\mu^G).
\end{equation}
Notice that the notation is justified as taking invariants $(\cdot)^G$ is a right exact functor since $G$ is reductive.

\begin{corollary} \label{lemma obstruction inclusion}
Let $P=k[W]$, let $I \subset P$ be the ideal of $X$, and let $N^G$ be the vector space defined by \eqref{eq def navecg}. Then there is an inclusion $\iota:\Ext^{1,G}_P(I,P/I) \into N^G$ of finite dimensional vector spaces, and thus the composition of the obstruction map defined by \eqref{ob map} with $\iota$ makes $N^G$ into an obstruction space for the deformation functor $D$.
\end{corollary}

\begin{proof}
First note that $\Hom_P^G(P^{n_2},P) \cong \Hom^G(N_2,P)$ is a $P^G$-module of finite type by \cite[Lemma 2.1]{Br10}, and thus so is $N^G$. Besides, it follows from the definition of $\mu$ that $N^G$ is supported on the image of $X=V(I)$ under the projection $\pi :W\to W/\!/G$, which is a scheme of finite length by \eqref{mod covariants}. Hence, $N^G$ is a finite dimensional vector space.

Let us now use the presentation \eqref{eq equivariant resolution} to calculate $\Ext^{1,G}_P(I,P/I)$. After adding a third step $P \otimes N_3\isom P^{n_3}$ to this presentation, we apply the functor $\Hom^{G}_P(\cdot,P/I)$ and obtain a complex
\begin{equation}\label{eq ext}
\Hom^{G}_P(P^{n_1},P/I)  \to[d_1] \Hom^{G}_P(P^{n_2},P/I)  \to[d_2] \Hom^{G}_P(P^{n_3},P/I)  
\end{equation}
such that $\Ext^{1,G}_P(I,P/I) = \Ker( d_2) / \Im (d_1)$.
As $P^{n_i}$ are free as $P$-modules, we get that $\Ext^{1,G}_P(P^{n_i},I)=0$, and thus the sequence \eqref{eq ext} extends to a commutative diagram with exact columns
\begin{equation}
\xymatrix{
0&0&0\\
\Hom^{G}_P(P^{n_1},P/I) \ar[u]\ar[r]^{d_1} & \Hom^{G}_P(P^{n_2},P/I)  \ar[u]\ar[r]^{d_2} & \Hom^{G}_P(P^{n_3},P/I)\ar[u] \\
\Hom^{G}_P(P^{n_1},P) \ar[u]\ar[r]^{D_1} & \Hom^{G}_P(P^{n_2},P)  \ar[u]^{E_1}\ar[r]^{D_2} & \Hom^{G}_P(P^{n_3},P)\ar[u] \\
\Hom^{G}_P(P^{n_1},P^{n_1}) \ar[u]\ar[r] & \Hom^{G}_P(P^{n_2},P^{n_1})  \ar[u]^{E_2}\ar[r] & \Hom^{G}_P(P^{n_3},P^{n_1})\ar[u] \\
}
\end{equation}
Let us consider the preimage $V\subset \Hom^{G}_P(P^{n_2},P)$ of $\Ker(d_2)$ under the surjective map $E_1$. We read off from the diagram that $\Ker(D_2)\subset V$, and that $D_1$ and $E_2$ both map to $V$. 
This immediately yields that
\[
\Ext^{1,G}_P(I,P/I) = \Coker\left(\Hom^{G}_P(P^{n_1},P)\oplus \Hom^{G}_P(P^{n_2},P^{n_1})  \to[D_1+E_2] V\right).
\]
But $D_1$ is given by composition with $V_0$ on the right, and $E_2$ is given by composition with $U_0$ on the left, thus $D_1+E_2$ is nothing else but $\mu^G$; see \eqref{eq obstruktionsraum}. As a consequence
\[
N^G = \Coker\left(\Hom^{G}_P(P^{n_1},P)\oplus \Hom^{G}_P(P^{n_2},P^{n_1})  \to[\mu^G] \Hom^{G}_P(P^{n_2},P) \right)
\]
contains $\Ext^{1,G}_P(I,P/I)$. This completes the proof. 
\end{proof}

\subsection{The obstruction map explicitly}\label{subsec obstruction explicit}
An  obstruction map for $N^G$ is thus obtained by composing the obstruction map defined by \eqref{ob map} with the inclusion $\iota$ of Corollary \ref{lemma obstruction inclusion}. It turns out that the obstruction map associated with $N^G$ is more suitable to do actual computations than the obstruction map given by \eqref{ob map}. Indeed, we can make the obstruction map for the obstruction space $N^G$ more explicit. For simplicity, we will do this only in the case of the universal deformation. With the notation at the end of section \ref{subsec tangent spaces}, this means that we consider a small extension 
\begin{equation}  \label{small extension univ}
 0\to  J\to S/\gothq \to R_n \to 0,
\end{equation}
where $R_n:=S/K_n = \Rch/\mm_{\Rch}^{n+1}$ is the base space of the $n$-th truncation ($n \geq 1$) of the universal deformation 
$$\lambda_n: \XX_n \to \Spec(R_n),$$ 
and $\gothq$ is an ideal of $S$ such that 
$$\mm_{S}^{n+2} \subset \gothq \subset K_n.$$ 

We suppose that we have calculated the universal deformation up to order $n$. In other words, we have generators for the ideal $K_n$, and we have a complex
\begin{equation}\label{eq resolution rn}
P_{R_n}^{n_2} \to[v_n] P_{R_n}^{n_1} \to[u_n] P_{R_n}
\end{equation}
such that $\XX_n=\Spec \left( P_{R_n}/\Im(u_n)\right)$, where $P_{R_n}=P \otimes R_n$ and $\Im(u_n)$ denotes the image of $u_n$.

We represent $u_n$ and $v_n$ by $G$-equivariant matrices
\begin{small}
\begin{equation}  \label{matrices U and V}
U_n=A_0 + \ldots + A_n \in \Hom_P(P^{n_1},P) \otimes S \quad \textrm{and } \quad V_n=B_0 + \ldots + B_n \in \Hom_P(P^{n_2},P^{n_1}) \otimes S
\end{equation}
\end{small}
such that the coefficients of $A_i$ and $B_i$ in $S$ are in $\Sym^i M$, that is, degree $i$ homogeneous polynomials.

We tensorize the sequence induced by the map $\mu^G$ defined by \eqref{eq obstruktionsraum equiv} with $S_{n+1}:=S/\mm_{S}^{n+2}$ to obtain the sequence
\begin{small}
\begin{equation}
S_{n+1} \otimes \left(  \Hom_P^G(P^{n_1},P) \oplus \Hom_P^G(P^{n_2},P^{n_1}) \right) \to S_{n+1} \otimes \Hom_P^G(P^{n_2},P) \to S_{n+1} \otimes N^G.
\end{equation}
\end{small}
Consider the element
\begin{equation}\label{eq obstruktion schritt n}
\gotho_{n+1} := - \sum_{p+q\leq n} A_p B_{q} - \sum_{i=1}^n A_i B_{n+1-i} \in S_{n+1} \otimes \Hom_P^G(P^{n_2},P),
\end{equation}
and denote by $\omega_{n+1}$ its image in $S_{n+1} \otimes N^G$.

\begin{theorem}\label{lemma cle} 
Let $n\geq 1$. With the notation above, the following hold:
\begin{enumerate}

\item\label{lemma cle enum 1} The obstruction map associated with the small extension \eqref{small extension univ} is given by 
$$\begin{array}{cccc}
ob  : &  D(R_n) & \to & J \otimes N^G \\
   &\lambda & \mapsto & (q \otimes Id)(\omega_{n+1}) 
\end{array}$$
where $q:S_{n+1} \to S/\gothq$ is the quotient map.

\item\label{lemma cle enum 2}
There exist $G$-equivariant matrices $A_{n+1}$ and $B_{n+1}$ with coefficients in $\Sym^{n+1} M$ such that $U_{n+1}:=U_n + A_{n+1}$ and $V_{n+1}:=V_n + B_{n+1}$ satisfy $U_{n+1} V_{n+1} \in \Hom_P(P^{n_2},P) \otimes K_{n+1}$. 
Every such couple represents the $(n+1)$-st truncation of the universal deformation $\lambda_{n+1}: \XX_{n+1} \to \Spec(R_{n+1})$.

\item\label{lemma cle enum 3} 
Write $\omega_{n+1}\in S_{n+1}\otimes N^G$ as $\sum_{i \in I} \overline{c_i} \otimes n_i$, where $\{n_i,\ i \in I\}$ is a basis of $N^G$, and each $\overline{c_i} \in S_{n+1}$. Let $c_i \in S$ be an arbitrary lift of $\overline{c_i}$. Then 
$$ K'_{n+1}:=(c_1,c_2,\ldots)+\mm^{n+2}=K_{n+1}.$$

\end{enumerate}
\end{theorem}

\begin{proof}

ad \eqref{lemma cle enum 1}: 
Let us check that $(q \otimes Id)(\omega_{n+1})$ coincides with the element $ob(\lambda)$ calculated in the proof of Proposition \ref{lemma ext1 obstruktionsraum}. We choose the morphisms $(u',v')$ in the middle row of Diagram \ref{lemma ext1 obstruktionsraum} to be residue classes of the matrices  $U_n$ and $V_n$ modulo $\gothq$. Then $\ob(\lambda)$ is represented by the image of $U_n V_n$ in $N^G \otimes (S/\gothq)$ which indeed up to a sign coincides with $(q \otimes Id)(\omega_{n+1})$.

ad \eqref{lemma cle enum 2}: 
By assumption $(U_n,V_n)$ represents the $n$-th truncation of the universal deformation $\lambda_n$. By Theorem \ref{eqMat}, as $\lambda_n$ is the restriction of $\lambda_{n+1}$, the sequence \eqref{eq resolution rn} lifts to $R_{n+1}$, that is, there exist $G$-equivariant matrices $U'_{n+1}\in \Hom_P(P^{n_1},P) \otimes S $ and $V'_{n+1}\in  \Hom_P(P^{n_2},P^{n_1}) \otimes S$ representing $\lambda_{n+1}$. Because of 
\[
U'_{n+1}\otimes R_n = u_n = U_{n}\otimes R_n \textrm{ and } V'_{n+1}\otimes R_n = v_n = V_{n}\otimes R_n,
\]
we may write
\[
U'_{n+1} - U_{n} = \kappa_1 + A_{n+1} \textrm{ and } V'_{n+1}-  V_{n} = \kappa_2 + B_{n+1},
\]
where the $\kappa_i$ have coefficients in $K_{n+1}$, and $A_{n+1}$ and $B_{n+1}$ have coefficients in $\Sym^{n+1}M$. Here we used that $K_n=K_{n+1}+\mm_S^{n+1}$. We may furthermore suppose that $\kappa_i$, $A_{n+1}$, and $B_{n+1}$ are $G$-equivariant by applying the Reynolds operator. So $U_{n+1}:= U'_{n+1}-\kappa_1$ and $V_{n+1}:= V'_{n+1}-\kappa_2$ satisfy $U_{n+1} V_{n+1} \in \Hom_P(P^{n_2},P) \otimes K_{n+1}$.

Let $I_{\ZZ_{n+1}}\subset P_{R_{n+1}}$ be the image of the map $P_{R_{n+1}}^{n_1}\to P_{R_{n+1}}$ induced by $U_{n+1}$. Then the matrices $U_{n+1}$ and $V_{n+1}$ determine an extension of \eqref{eq resolution rn} over $R_{n+1}$ so that $P_{R_{n+1}}/I_{\ZZ_{n+1}}$ is flat over $R_{n+1}$ by item \eqref{thm enum lifting relations} of Theorem \ref{eqMat}. It then remains to show that any two extensions of the $n$-th truncation of the universal deformation over $R_{n+1}$ are isomorphic. Thinking in terms of classifying morphisms we have to show that a morphism $\varphi: R_{n+1}\to R_{n+1}$ which induces the identity on $R_n$ is an isomorphism. As $n \geq 1$, the result follows from Lemma \ref{lemma fact}.

ad \eqref{lemma cle enum 3}: 
From \eqref{lemma cle enum 1} we obtain that $\omega_{n+1} \otimes_S R_{n+1}=0$ so that $K'_{n+1} \subset K_{n+1}$. By the same argument as in \eqref{lemma cle enum 2}, we can lift $\lambda_n \in D(R_n)$ to some $\lambda'_{n+1} \in D(S/K'_{n+1})$ because $\omega_{n+1}$ vanishes modulo $K'_{n+1}$. This endows us with a morphism $\phi:R_{n+1} \to S/K'_{n+1}$ induced by the classifying morphism $\Rch \to S/K'_{n+1}$. Then, as the $(n+1)$-st truncation of the universal deformation is the biggest $(n+1)$-st order deformation of $X$ with tangent space of dimension $\dim (T^1)$, we obtain the inclusion $K_{n+1}\subset K'_{n+1}$. To see this last claim, consider the commutative diagram
\begin{equation}\label{eq finally}
\xymatrix{
R_{n+1} \ar[r]^{\phi} \ar[d] &S/K'_{n+1} \ar[r]^(0.4){q} & R_{n+1}=S/K_{n+1} \ar[d]\\
R_{n} \ar[rr]_{\id} & &  R_n\\
}
\end{equation}
where $q$ and the unlabeled arrows are the canonical quotient maps, and $\phi$ is the map given by the universal property of $\Rch$.

Since $q \circ \phi$ induces an automorphism on the cotangent space, it follows from Lemma \ref{lemma fact} that $q \circ \phi$ is an automorphism. In particular, $\phi$ is injective, and thus $\dim(S/K_{n+1}) \leq \dim(S/K'_{n+1})$. As $K_{n+1}\subset K'_{n+1}$ and both contain $\mm^{n+2}$, we get that $K_{n+1}=K'_{n+1}$.
\end{proof}

\begin{remark} ---

\begin{itemize}
\item Theorem \ref{lemma cle} suggests a way to compute the ideal $K_n$ for any $n$; see section \ref{section algorithm} for an algorithmic description. Notice however that here the modus operandi is somewhat different to what we do in theory. Theoretically, we would pick an ideal $\gothq$ which is far bigger than $\mm_S^{n+2}$ (e.g. $\gothq=\mm_S.K_n$), and let the obstruction theory produce additional equations $f_1,f_2,\ldots$ In practice, we let an adapted version of obstruction theory take us to the ideal directly.

\item The search for the matrices $A_{n+1}\in \Hom_P(P^{n_1},P) \otimes \Sym^{n+1} M$ and $B_{n+1}\in \Hom_P(P^{n_2},P^{n_1}) \otimes \Sym^{n+1} M$ in the proof of Theorem \ref{lemma cle} \eqref{lemma cle enum 2} corresponds to the search for the elements $\xi$ and $\delta$ in the proof of Proposition \ref{lemma ext1 obstruktionsraum}. 
\end{itemize}
\end{remark}

Strictly speaking we do not need the obstruction space $\Ext^{1,G}_P(I,P/I)$ of Proposition \ref{lemma ext1 obstruktionsraum} as we handle the obstruction theory directly in our space $N^G$. However, the construction of the latter is strongly motivated by the former, which becomes clear in the proof of Corollary \ref{lemma obstruction inclusion}, and so we considered it worthwhile to include it. Furthermore, it might also be relevant to practice. There are examples with $\Ext^{1,G}_P(I,P/I)=0$ where $N^G\neq 0$ thus furnishing a way to prove unobstructedness results. Such a situation occurs for instance in \cite{CF}.

\section{The case of an extra \texorpdfstring{$\Gm$}{Gm}-action}\label{subsec gm action}

Theorem \ref{lemma cle} suggests an algorithm to calculate the truncation of the universal deformation of a point  $[X]\in \HH$ up to arbitrarily high order. It will be described in section \ref{section algorithm}. In general, this procedure will never stop as the ideal $\hat K$ defined by \eqref{eq def khut} may be not generated by polynomials. This is different in the presence of an extra $\Gm$-action with strictly positive weights; see Theorem \ref{thm open immersion}. Such an action has even more important practical consequences, namely it guarantees that our algorithm stops; see section \ref{subsec stop} for the stop condition.

In the setting of section \ref{subsec setup}, we make the following extra hypothesis on the point $[X]\in \HH$ under consideration:
\begin{hypo}\label{hypo weights} 
There is a subgroup of $\Aut^G(W)$, isomorphic to the multiplicative group $\Gm$, such that:
\begin{itemize}
\item $\Gm$ acts on $P=k[W]$ with strictly positive weights; and
\item $[X] \in \HH$ is a $\Gm$-fixed point, and the induced action on the cotangent space $(T^1)^\vee=(T_{[X]} \HH)^\vee$ has strictly positive weights.
\end{itemize}
\end{hypo}

\begin{lemma}\label{rmk equivariant}
Under Hypothesis \ref{hypo weights}, the morphism $\psi:\Shut \to \Rch$ of Lemma \ref{eq choice} can be chosen $\Gm$-equivariant. From now on, we will assume that this is the case. Then the ideal $\hat K=\ker (\psi)$ is generated by weight vectors for the $\Gm$-action on $S$.
\end{lemma}
\begin{proof}
Recall that $\psi$ depends on the choice of $y_1,\ldots, y_d \in \mm_{\Rch}$; we may choose the $y_i$ to be weight vectors for the $\Gm$-action on $\Rch$. Then so are their images $t_1,\ldots,t_d \in \mm/\mm^2 \cong M$, and thus the map $\psi: t_i \mapsto y_i$ is $\Gm$-equivariant. This implies that the ideal $\hat K=\Ker(\psi)$ is $\Gm$-stable. We take weight vectors $f_1, \ldots, f_k \in \hat K$ such that their images generate the finite dimensional vector space $\hat K/\mm \hat K$. By Nakayama's Lemma, the $f_i$ generate $\hat K$ and as $\Gm$ acts with strictly positive weights on $M$, all $f_i$ are contained in $\Sym^{\leq N}M$ for some $N\gg 0$. See \cite[Lemma A.4]{Na08} for a constructive proof.
\end{proof}

Denoting $K:=\hat K \cap S$, we see that the base space of the universal deformation $\XX \to \Spec(\Shut/\hat K)$ is the completion of a finite type scheme, namely that of $\Spec(S/ K)$. We knew this before: it is also the completion of an affine neighborhood of $[X]$ in $\HH$, which has already been used in the proof of Lemma \ref{repD}. But the algebraization coming from deformation theory can be calculated algorithmically; see section \ref{section algorithm}. 
The next result tells us that we can use deformation theory to calculate also an algebraization of the universal deformation.

\begin{lemma} \label{lem homogeneous}
In the setting of section \ref{subsec setup} and under Hypothesis \ref{hypo weights}, there exist $\alpha \geq 1$ and $G\times \Gm$-equivariant matrices 
$$(U_{\alpha}, V_{\alpha}) \in ( \Hom_P(P^{n_1},P) \otimes \Sym^{\leq \alpha}M) \times ( \Hom_P(P^{n_2},P^{n_1}) \otimes \Sym^{\leq \alpha}M) $$ 
which represent the universal deformation $\XX \to \Spec(\Shut/\hat K)$.
\end{lemma}

\begin{proof}

It follows from Hypothesis \ref{hypo weights}  and Proposition \ref{groupaction} that the universal deformation is $\Gm$-equivariant. Hence, for any $n \geq 0$, the $n$-th truncation of the universal deformation is also $\Gm$-equivariant. By Theorem \ref{lemma cle}, there exist $(G\times \Gm)$-equivariant matrices $(U_n,V_n)$ with coefficients in $\Sym^{\leq n}M$ such that $U_n V_n \in \Hom_P(P^{n_2},P) \otimes K_{n}$. If we decompose $U_n$ and $V_n$ into graded pieces for $S=\Sym^\bullet M$ as in \eqref{matrices U and V}, this amounts to saying that all the $A_k$ and $B_k$ are $(G \times \Gm)$-equivariant. Moreover, the $G$-modules $N_1$ and $N_2$ defined at the beginning of section \ref{G eq reso} may be chosen $\Gm$-stable. Let us fix bases $\{f_1,\ldots,f_{n_1}\}$ and $\{r_1,\ldots,r_{n_2}\}$ of $N_1$ and $N_2$ respectively such that each $f_i$ respectively each $r_j$ is a weight vector for the $\Gm$-action. Moreover, by Lemma \ref{rmk equivariant} the system of parameters $t_1,\ldots,t_d \in M$ of $S$ is chosen to consist of weight vectors as well. Then the matrices $A_k$ are $\Gm$-equivariant if and only if for every $1 \leq i \leq n_1$ the $i$-th entry of $A_k$ is either $0$ or a weight vector for the $\Gm$-action with the same weight as that of $f_i$. In particular, as the weight of of the $t_i$-monomials in $A_k$ increases, the weight of the $P$ variables has to decrease. As the coefficients of $A_n$ belong to $P \otimes \Sym^k M$, Hypothesis \ref{hypo weights}  implies that 
$$n \leq \text{weight of any non-zero coefficient of $A_n$} \leq \alpha_1:=\max_{l=1,\ldots,n_1} \weight(f_l).$$  
Arguing similarly for $B_n$, we also obtain that
$$n \leq \text{weight of any non-zero coefficient of $B_n$} \leq \alpha_2:=\max_{\substack{p=1,\ldots, n_1\\ q=1,\ldots, n_2 }} \weight(f^*_p \otimes r_q),$$
where $\{f_1^*,\ldots,f_{n_1}^*\}$ stands for the dual basis to $\{f_1,\ldots,f_{n_1}\}$.  
Consequently, for every $n > \alpha:=\max(\alpha_1,\alpha_2)$, we have $A_n=B_n=0$, whence the result.
\end{proof}

\begin{remark}
The proof of Lemma \ref{lem homogeneous} gives an explicit bound for $\alpha$. Moreover, Lemma \ref{rmk equivariant} and Theorem \ref{lemma cle} imply that the ideal $K$ is generated by weight vectors whose weight is bounded by $2 \alpha w_0$, where $w_0$ denotes the maximal weight of $(T^1)^\vee$ for the $\Gm$-action.
\end{remark}

Let us denote by $U$ and $V$ the matrices $U_{\alpha}$ and $V_{\alpha}$ of Lemma \ref{lem homogeneous}; these matrices have coefficients in $S$ and satisfy $U  V = 0$ mod $K$. Hence, denoting $\overline{U} \in \Hom_P(P^{n_1},P) \otimes (S/K)$ and $\overline{V} \in \Hom_P(P^{n_2},P^{n_1}) \otimes (S/K)$ the residue classes of  $U$ and $V$, we get that $\overline{U} \overline{V}=0$. The family 
\[
\ZZ:=\Spec \left( (P \otimes (S/K))/\left\langle \Im (\overline{U})\right\rangle  \right) \to [\phi] E:=\Spec(S/K)
\]
is flat over $0 \in E$ by construction. Hence, $\phi$ is flat over an open subset containing $0$ by general theory, and the $\Gm$-action allows to deduce flatness everywhere. Then it follows from the universal property of the invariant Hilbert scheme that there exists a $\Gm$-equivariant diagram
\begin{equation} \label{def_iota}
\xymatrix{
\ZZ \ar[r]\ar[d]_\phi & \UU \ar[d]\\
 E\ar[r]^\iota & \HH\\
 }
 \end{equation}
such that $\ZZ=\UU \times_{\HH} E$, where $\UU \to \HH$ is the universal family.  

We want to show that $\iota$ is an open immersion, and that hence $\ZZ$ is the restriction of the universal family. For this  will need the following
 
\begin{lemma}\label{lemma namikawa}
Let $A=\bigoplus_{n\geq 0} A_n$ be a graded noetherian $k$-algebra with $A_0=k$. Let $\mm=\bigoplus_{n\geq 1} A_n$ be the irrelevant maximal ideal. Then the completion of $A$ at $\mm$ is
\[
\widehat A :=  \varprojlim_n A/\mm^n = \prod_{n \geq 0} A_n =: A'.
\]
In particular, $A$ is the subalgebra of $\hat{A}$ generated by the homogeneous elements. 
\end{lemma}

\begin{proof}
Consider the maximal ideal $\mm':=\prod_{n \geq 1} A_n$ of $A'$. In virtue of the equality $A'/(\mm')^n=A/\mm^n$, the universal property of the inverse limit endows us with a map $A'\to \widehat A$. We will argue that it is bijective. The definition of $\mm$ implies that $\mm^k\subset A_{\geq k}:=\bigoplus_{n\geq k} A_n$. 
Take some $f=(\bar f_n)_{n\in \NN}\in \widehat A$ where $\bar f_n\in A/\mm^n$ and denote by $f_n \in A$ an arbitrary lift of $\bar f_n$ to $A$. We decompose it in graded pieces 
$$ f_n = f_n^0 + f_n^1 + f_n^2 + \ldots $$
with $f_n^d\in A_d$. As $f_n - f_{n+m}$ is contained in $\mm^n \subset A_{\geq n}$ for all $m$ we have that $f_n^d=f_{n+m}^d$ for all $d<n$. If we denote $f^d:=f_{d+1}^d=f_{d+2}^d=f_{d+3}^d=\ldots$, then the tuple $(f^0, f^1,\ldots)\in A'$ is a preimage of $f$, so we get the surjectivity. 

A morphism to $\Ach = \varprojlim_k A/\mm^k$ is certainly injective if one of the induced maps to the $A/\mm^k$ is. A homogeneous morphism with source $A'$ is injective if it is so on each graded piece $A_n$. Injectivity is thus a consequence of the fact that $A_n\to \widehat A \to A/\mm^{n+1}$ is injective, where we again used that $\mm^{n+1}\subset A_{\geq n+1}$. The last statement of the lemma is clear.
\end{proof}

\begin{remark}
Lemma \ref{lemma namikawa} is the converse of \cite[Lemma A.2]{Na08}, but much easier. In general, our presentation and arguments have been inspired by \cite[Appendix A]{Na08}.
\end{remark}

\begin{theorem} \label{thm open immersion}
The morphism $\iota$ defined by \eqref{def_iota} is an open immersion.
\end{theorem}

\begin{proof}
It is sufficient to  show that $\iota$ is \'{e}tale and injective. 
By construction, it is \'{e}tale over the distinguished point $0\in E$ and, by general theory, over a Zariski open subset $U$ containing $0$.
Let us note that the existence of the $\Gm$-action with a unique fixed point implies that every $\Gm$-orbit of each point in $E$ meets $U$. Thus, $\iota$ is \'{e}tale everywhere and its  image $V:=\iota(E) \subset \HH$ is Zariski open in $\HH$.

Suppose that $V$ is affine, say $V = \Spec(A)$. Let us write $R:=S/K$, and let $\iota^\# : A \to R$ be the induced map. By construction of $R$ and by definition of $A$, the map induced by $\iota^\#$ between the completions at the unique $\Gm$-stable maximal ideals $\widehat A = \OH \to \Rch$ is an isomorphism and, by  Lemma \ref{rmk equivariant}, it is $\Gm$-equivariant.
The claim now follows from  Lemma \ref{lemma namikawa} as $A$ and $R$ can be recovered from $\widehat A$ and $\Rch$ as the $k$-algebras generated by the eigenvectors for the $\Gm$-action. 

Let us now consider the general case. As $0\in E$ is the only $\Gm$-fixed point in $E$, and as $\iota$ is \'{e}tale, it follows that $v:=\iota(0)$ is the only $\Gm$-fixed point in $V$, and that $0$ is the only preimage of $v$. For $x \in E\ohne \{0\}$, the orbit $\Gm\cdot \iota(x)$ is thus  one dimensional. Let $Z$ denote its Zariski closure in $V$. Set-theoretically, $Z= \Gm\cdot \iota(x) \cup \{v\}$ as $v$ is the only $\Gm$-fixed point in $V$. The restriction $\iota\vert_{\iota^{-1}(Z)}:\iota^{-1}(Z)\to Z$ is $\Gm$-equivariant, \'{e}tale and its completion at the $\Gm$-fixed points is an isomorphism. If $Z$ is affine, the same argument as above gives the bijectivity.

Suppose that $Z$ is not affine; thus it is a projective singular rational curve. 
Clearly $Z$ is singular, since on a smooth complete curve, the complement of a $\Gm$-orbit cannot be a single point. Let $\nu: \widetilde{Z} \to Z$ be the normalization. The action of $\Gm$ extends to $\widetilde{Z}$ and $\nu$ is equivariant so that, by the set-theoretical description of $Z$, we have $\nu^{-1}(v)=\{x_+,x_-\}$. Consider the map between the completions
\[
\nu^\#:\widehat \sO_{Z,v} \to \widehat\sO_{\widetilde{Z},x_+}\times \widehat\sO_{\widetilde{Z},x_-}.
\]
The map $\nu^\#$ is injective and equivariant. Let $t_+$ and $t_-$ denote local parameters at $x_+$ and $x_-$ respectively, which are weight vectors for the $\Gm$-action. Let $\mm_v$ be the maximal ideal of $v$ in $\sO_{Z,v}$. Then $\Gm$ acts on $\mm_v/\mm_v^2$ with strictly positive weights by Hypothesis \ref{hypo weights}. Both maps $\sO_{Z,v} \to \sO_{\widetilde{Z},x_\pm}$ are non-zero, thus the weights of $t_-$ and $t_+$ for the $\Gm$-action are both strictly positive.
But $(\widetilde{Z},x_+,x_-)$ is isomorphic to $(\PP^1,0,\infty)$, and on the latter it is an easy exercise to check that there is no $\Gm$-action with strictly positive weights on both points $0$ and $\infty$. This is a contradiction, and hence $Z$ cannot be projective, completing the proof.
\end{proof}

Sometimes it happens that we do not find a $\Gm$-action with positive weights on the whole of $M=(T^1)^\vee$. For a given $\Gm$-action as above with positive weights on $k[W]$ and arbitrary weights on $M$, let $M_1\subset M$ be the subspace generated by vectors of positive weight. The unique $\Gm$-equivariant projection $M\to M_1$ which is the identity on $M_1$ gives rise to a subscheme of the base space of the universal deformation as follows.

\begin{corollary} \label{cor open immersion}
There exists a $\Gm$-scheme $E_1$ of finite type over $k$, and a $\Gm$-equivariant morphism $\iota_1:E_1\to \HH$ such that the following statements hold.
\begin{enumerate}
\item The Zariski cotangent space to $E_1$ at $0$ is $M_1$;
\item $\iota_1(0)= [X]$;
\item $\iota_1$ is an immersion, that is, the composition of an open immersion and a closed immersion; and
\item $\iota_1$ induces the projection $M\to M_1$ on cotangent spaces.
\end{enumerate}
\end{corollary}
\begin{proof}
Let $S_1:= \Sym^\bullet M_1$, let $\widehat S_1$ be its completion at the maximal ideal generated by $M_1$, let $\Shut\to \widehat S_1$ be the induced surjection, and let $N_1$ be its kernel. Then the arguments from Lemma \ref{rmk equivariant} show that $N_1 + \hat K$ is generated by weight vectors for the $\Gm$-action. Put $K_1:=S_1 \cap (N_1 + \hat K)$ and $E_1:=\Spec(S_1/K_1)$. The maximal ideal generated by the image of $M_1$ in $S_1/K_1$ defines the point $0\in E_1$. Lemma \ref{lem homogeneous} carries over mutatis mutandis and the resulting family of subschemes of $W$ gives a $\Gm$-equivariant morphism $\iota_1:E_1\to \HH$ which certainly sends $0$ to $[X]$. Finally, the proof of Theorem \ref{thm open immersion} goes through literally with $\overline{\iota_1(E_1)}$ in place of $\HH$.
\end{proof}

\section{The algorithm}\label{section algorithm}

In this section we describe our algorithm (Algorithm \ref{algo main}) to calculate the universal deformation. The  proof of its validity follows from Theorem \ref{lemma cle} and from section \ref{subsec gm action}. We will make the link when necessary and focus mostly on how to perform the steps of the algorithm in practice. For our calculations, we used the computer algebra system \cite[Macaulay2]{Mac2}.

For $i=1,2$, we identify $P \otimes N_i$ with $P^{n_i}$ as before; see Notation \ref{not avec base}. Moreover, to simplify the notation, we will denote by $P^{a \times b}$ the space $\Hom_P(P^a,P^b)$ of matrices with $b$ rows and $a$ columns.

\subsection{The presentation of \texorpdfstring{$P/I$}{P/I}}\label{subsec resolution}

Recall the setting (see sections \ref{subsec setup} and \ref{subsec gm action}): we consider a reductive algebraic group $G$, an affine $G$-scheme of finite type $W$, and we want to deform a $(G \times \Gm)$-stable closed subscheme $X\subset W$, where $\Gm$ is a subgroup of $\Aut^G(W)$ satisfying Hypothesis \ref{hypo weights}.
We denote by $P$ the coordinate ring of $W$, and by $I\subset P$ the defining ideal of $X$.  In particular, $I$ is a $(P,G \times \Gm)$-module.

Recall the definition of $K_n$, $S$, and $M$ from section \ref{subsec tangent spaces}. For a given $n \geq 0$, we have to find the ideal $K_n$ such that $S/K_n$ is the $n$-th truncation of the base space of the universal deformation, and matrices
\[
A_i \in {(P^{1\times n_1}\otimes \Sym^i M)}^{G \times \Gm} \textrm{ \ \ and \ \  } B_i \in {(P^{n_1\times n_2}\otimes \Sym^i M)}^{G \times \Gm} 
\]
such that $U_n = \sum_{i=0}^n A_i$ and $V_n = \sum_{i=0}^n V_i$ represent the universal deformation up to order $n$. 

Let us start with $n=0$. We choose a minimal dimensional $(G \times \Gm)$-submodule $N_1 \subset P$ which generates the ideal $I$, and we denote by $\{f_1,\ldots,f_{n_1}\}$ a basis of $N_1$. 
We fix a minimal $(G \times \Gm)$-submodule $N_2\subset P \otimes N_1$, which gives rise to an exact sequence of $(P,G \times \Gm)$-modules
$$ P \otimes N_2 \to P \otimes N_1 \to I \to 0,$$
and we denote by $\{r_1,\ldots,r_{n_2}\}$ a basis of $N_2$. Here we choose bases of $N_1$ and $N_2$ with respect to their decomposition into irreducible $G$-modules, and such that the $f_i$ and the $r_j$ are weight vectors for the $\Gm$-action. Then there is a unique decomposition $r_j=\sum_{i=1}^{n_2} p_{ij} \otimes f_i$, where $p_{ij} \in P$. 
One easily checks that the matrices
$$U_0=A_0= \begin{bmatrix} f_1 & f_2 & \ldots & f_{n_1} \end{bmatrix} \in  {(P^{1\times n_1})}^{G \times \Gm}$$ 
and
$$V_0=B_0=\begin{bmatrix} p_{11} & p_{12} & \ldots & p_{1n_2} \\
p_{21} & p_{22} & \ldots & p_{2n_2} \\
\vdots &    \vdots      &       & \vdots \\
p_{n_11} & p_{n_12} & \ldots & p_{n_1n_2}         
\end{bmatrix} \in {(P^{n_1\times n_2})}^{G \times \Gm} $$
give the presentation.

The computation of $V_0$ for a given $U_0$ can be done with \cite[Macaulay2]{Mac2} since $V_0$ is nothing else than the first syzygy matrix of the ideal $I$ with respect to the generators $f_1,\ldots,f_{n_1}$. To be precise, \cite[Macaulay2]{Mac2} provides a vector subspace $N'_2 \subset P \otimes N_1$ which generates the kernel of the multiplication map $P\otimes N_1\to P$ as a $P$-module, but may not be $G$-stable. In that case, we take for $N_2$ the $G$-submodule of $P\otimes N_1$ generated by $N'_2$.

\subsection{The first order deformation}  \label{first order}
  
Recall from section \ref{subsec tangent spaces} that $T^1=\Hom_{P}^G(I,P/I)$ is the space of first order $G$-stable deformations of $X$ inside $W$. Let us explain how to compute a basis of $T^1$.

First, we determine a basis of the vector space 
$$V:=\Hom_P^G(P \otimes N_1,P/I) \cong \Hom^G(N_1,P/I).$$

This is done as follows. Let $N_1=\bigoplus_{j \in J} M_j^{\oplus m_j}$ be a decomposition of $N_1$ into irreducible $G$-modules. By definition of $h$, we have $P/I\isom\bigoplus_{M \in \Irr(G)} M^{\oplus h(M)}$ as a $G$-module, so the dimension of $V$ is given by
$$D:=\dim(V)=\dim \left( \Hom^G \left(\bigoplus_{j \in J} M_j^{\oplus m_j}, \bigoplus_{M \in \Irr(M)} M^{\oplus h(M)} \right)\right)=\sum_{j \in J} m_j h(M_j).$$
An explicit basis of $V$ can be obtained as follows: 
\begin{enumerate}
\item For $1 \leq i \leq n_1$, denote by $E_i \in P^{1 \times n_1}=\Hom_P(P \otimes N_1,P)$ the matrix whose $i$-th coefficient is $1$ and all the others are $0$.
\item Compute $\LL_0:=\{ \pi \circ \RR(E_1),\ldots, \pi \circ \RR(E_{n_1}) \} \subset \Hom_P^G(P \otimes N_1,P/I)$, where $\RR: \Hom_P(N_1 \otimes P,P) \to \Hom_P^G(N_1 \otimes P,P)$ is the Reynolds operator, and $\pi: P \to P/I$ is the quotient map. Extract a basis $\BB_0$ of the vector space generated by $\LL_0$. If $\Card(\BB_0)=D$, then $\BB_0$ is a basis of $V$, else go to the next step.
\item Fix a basis $\{p_1^1,\ldots,p_{k_1}^1 \}$ of the $\Gm$-submodule $P_1$ of $P$ generated by weight vectors of weight $1$. Note that, for every $i \geq 0$, the vector space $P_i$ is finite-dimensional by Hypothesis \ref{hypo weights}. Compute $\LL_1:=\{ \pi \circ \RR(p_j^1 E_l), 1 \leq j \leq k_1, 1 \leq l \leq n_1 \}$, and extract a basis $\BB_1$ of the vector space generated by $\LL_0 \cup \LL_1$. If $\Card(\BB_1)=D$, then $\BB_1$ is a basis of $V$, else go to the next step.
\item Fix a basis $\{p_1^2,\ldots,p_{k_2}^2 \}$ of $P_2 \subset P$, compute $\LL_2$, extract a basis $\BB_2$ of $\LL_2 \cup \LL_1 \cup \LL_0$ etc. 
\end{enumerate} 

Since $V$ is finite-dimensional, this procedure has to stop after a finite number of steps. Unfortunately, we were unable to get an upper bound for the number of steps.

Once we have a basis $\BB_{N}=\{\vv_1,\ldots,\vv_D\}$ of $V$, it is easy to determine a basis of $T^1$ seen as a vector subspace of $V$. Indeed, we have seen in section \ref{G eq reso} that $T^1$ is just the kernel of the linear map
\[
V=\Hom_P^G(P \otimes N_1,P/I) \to \Hom_P^G(P \otimes N_2,P/I),\ v  \mapsto v B_0.
\]  
Let us denote $\BB=\{ \sss_1, \ldots, \sss_d\} \subset (P^{1\times n_1})^{G}$ such that $\{ \pi \circ \sss_1,\ldots, \pi \circ \sss_d\}$ is a basis of $T^1$. At this point, we can assume that each $\sss_i$ is a weight vector for the $\Gm$-action. Then we denote by $\{ t_1,\ldots,t_d\}$ the dual basis to $\BB$; in particular, each $t_i$ is also a weight vector for the $\Gm$-action. Note that
$$\Rch=k[\![t_1,\ldots, t_d]\!]/\hat K,$$ 
where $\hat K$ is the ideal defined by \eqref{eq def khut}, and
\[
A_1 = \sum_{i=1}^d t_i \sss_i \in {(P^{1\times n_1}\otimes M)}^{G \times \Gm}.
\]

To find $B_1$, we have to solve $A_0 B_1 = -A_1 B_0$. This equation has a solution because, by construction of $A_1$, the diagram
\[
\xymatrix{
P^{n_2} \ar[r]^{B_0} &P^{n_1}\ar[d]_{A_0}\ar[r]^{A_1}  & P\otimes M \ar[d]\\
&I \ar[r]_(0.3){\sum_i t_i^\vee\otimes t_i} & (P/I)\otimes M\\
}
\]
commutes. Denoting $B'_1 \in  P^{n_1\times n_2}\otimes M$ such that $A_0 B'_1=-A_1 B_0$, it follows from Corollary \ref{cor reynolds} that $B_1:=\RR(B'_1) \in (P^{n_1\times n_2}\otimes M)^{G \times \Gm}$ does the job, where $\RR$ is the Reynolds operator.

All these steps can be performed with \cite[Macaulay2]{Mac2} or any other computer algebra system.

\subsection{The higher order deformations}

The algorithm we perform is stipulated by Theorem \ref{lemma cle}. Suppose that, for some $n \geq 1$, we have calculated $(G \times \Gm)$-equivariant matrices
\[
U_n = A_0+\ldots+ A_n \quad\textrm{ and } \quad V_n = B_0+\ldots+ B_n
\]
as well as the ideal $K_n = K+\mm_S^{n+1}$ such that $U_n V_n =0$ modulo $K_n$. We will perform the step $n+1$ of the algorithm, that is, the computation of $A_{n+1}$, $B_{n+1}$ and $K_{n+1}$. After each step, we check for the stop condition (see section \ref{subsec stop}) and perform the next step if the stop condition is not satisfied.

\subsection{A stop condition}\label{subsec stop}

By Lemma \ref{rmk equivariant}, the ideal $\hat K \subset \Shut$ is generated by weight vectors for the $\Gm$-action on $S$. At the step $n$ of the algorithm, we obtain $K_n$ by calculating a list of weight vectors $g_i^n\in S$ with strictly positive weight
\[
g_1^n, \ldots, g_k^n \in \Sym^{\leq n} M
\]
such that 
\[
K_n=(g_1^n, \ldots, g_k^n ) + \mm_S^{n+1}.
\]
Take $d_i^n$ such that $g_i^n \in \mm^{d_i^n}$ but $g_i^n \notin \mm^{1+d_i^n}$. As all weights for the $\Gm$-action on $S$ are strictly positive by Hypothesis \ref{hypo weights}, the $g_i^n$ will be in $K$ if
\begin{equation}\label{eq stop condition 0}
\weight(g_i^n) < n\cdot \min_{i=1}^d \weight(t_i).
\end{equation}
Let us denote by $K'_n$ the ideal of $S$ generated by the $g_i^n$ satisfying the condition \eqref{eq stop condition 0}. Then the stop condition is that
\begin{equation}\label{eq stop condition}
U_n V_n= 0 \ \mod K'_n
\end{equation}
holds. By Lemma \ref{lem homogeneous}, there exists $\alpha$ such that $U_n V_n= 0 \mod K$ for every $n \geq \alpha$. Moreover, by Lemma \ref{rmk equivariant} and by definition of $K'_n$, there exists $\beta$ such that $K'_n=K$ for every $n \geq \beta$. Hence, the condition \eqref{eq stop condition} is satisfied for every $n \geq \max(\alpha,\beta)$.

Let $n_0 \geq 1$ be minimal such that the condition \eqref{eq stop condition} holds. Then one may check that $(U_{n_0},V_{n_0})$ represent the universal deformation (use Theorem \ref{lemma cle} \eqref{lemma cle enum 2}), and that $K=K'_{n_0}$ (use Theorem \ref{lemma cle} \eqref{lemma cle enum 3}). 
In particular, if $A_1B_1=0$, then $K=\{0\}$.

\subsection{The Algorithm}  \label{code_algo}

The input of the algorithm is an ideal $I \subset P=k[W]$ which is $(G \times \Gm)$-stable and satisfies the assumptions of section \ref{subsec setup} and Hypothesis \ref{hypo weights}.
The output is a quadruplet $(S,K,U,V)$ where 
\begin{itemize}
\item $S$ is the polynomial ring defined by \eqref{def S}; 
\item $K \subset S$ is the ideal such that $S/K$ is the base space of the universal deformation; and
\item $(U,V) \in  \left(P^{1\times n_1} \otimes S \right)^{G \times \Gm} \times \left( P^{n_1\times n_2} \otimes S\right)^{G \times \Gm}$ is a couple of matrices representing the universal deformation $\XX \to \Spec(S/K)$. 
\end{itemize}

\begin{algo}\label{algo main} ---

\begin{enumerate}
\item Input: the ideal $I=(f_1,\ldots,f_{n_1}) \subset P$, where the $f_i$ are as in section \ref{subsec resolution}
\item $A_0:=\begin{bmatrix} f_1 & \cdots & f_{n_1} \end{bmatrix} \in P^{1\times n_1} $
\item Compute the syzygy matrix $B_0 \in P^{n_1\times n_2}$ of the ideal $I$
\item \label{debut tang} For $1 \leq i \leq n_1$, denote by $E_i \in P^{1\times n_1}$ the matrix whose $i$-th coefficient is 1 and the other coefficients are 0
\item $m:=0$
\item \label{step vorher} Fix a basis $\{p_1^m,\ldots,p_{k_m}^m \}$ of the vector subspace $P_{m} \subset P$ generated by weight vectors of weight $m$ for the $\Gm$-action
\item \label{hard} Compute $\LL_m:=\{\pi \circ \RR(p_j^m E_l), 1 \leq j \leq k_m, 1 \leq l \leq n_1 \}$, where $\RR: P^{1\times n_1}  \to (P^{1\times n_1})^G$ is the Reynolds operator, and $\pi: P \to P/I$ is the quotient map
\item If $\rk(\LL_m \cup \ldots \cup \LL_0)<\dim(\Hom^G(N_1,P/I))$, then $m:=m+1$ and go to Step \eqref{step vorher}  
\item \label{fin tang} Compute a basis $\BB \subset \LL_m$ of the  kernel of the linear map 
$$\Hom_P^G(N_1 \otimes P,P/I) \to \Hom_P^G(N_2 \otimes P,P/I),\ \vv  \mapsto \vv B_0$$
\item Represent elements of $\BB$ by matrices $\sss_1, \ldots, \sss_d$ in $(P^{1\times n_1})^{G \times \Gm}$
\item $M:=\left \langle t_1,\ldots,t_d \right \rangle$
\item $A_1:=\sum_{i=1}^{d} t_i \sss_i \in (P^{1\times n_1} \otimes M)^{G \times \Gm}$
\item Solve the matrix equation $A_0X=-A_1B_0$
\item\label{rey2} $B_1:=\RR(X)  \in  (P^{n_1\times n_2} \otimes M)^{G \times \Gm}$
\item $U_1:=A_0+A_1$, $V_1:=B_0+B_1$
\item $n:=2$
\item\label{goto} $P_n:=-\sum_{p+q\leq n-1} A_p B_{q} - \sum_{i=1}^{n-1} A_i B_{n-i} \in P^{1\times n_2} \otimes \Sym^{\leq n} M$
\item Compute the image $\omega_n$ of $P_n$ in $\Coker(\mu)$, where 
$$\begin{array}{cccc}
\mu:  & \left( P^{1\times n_1} \otimes {\Sym}^{\bullet} M \right) \oplus \left( P^{n_1\times n_2}\otimes {\Sym}^{\bullet} M \right) &   \to  & P^{1\times n_2} \otimes {\Sym}^{\bullet} M \\ 
    &  (X_1,X_2) &  \mapsto & X_1 B_0 + A_0 X_2 
\end{array}$$
\item Take $c_1,\ldots,c_r$ the coefficients of $\omega_n$ in ${\Sym}^{\leq n} M$ 
\item $K_n:=(c_1,\ldots,c_r)$
\item Solve the matrix equation $A_0X_1+X_2B_0=P_n$ mod $K_n$
\item\label{rey3} $A_n:=\RR(X_2) \in  (P^{1\times n_1} \otimes \Sym^n M)^{G \times \Gm}$, $B_n:=\RR(X_1) \in (P^{n_1\times n_2} \otimes \Sym^n M)^{G \times \Gm}$
\item $U_n:=U_{n-1}+A_n$, $V_n:=V_{n-1}+B_n$
\item $K'_n:=(c_{i_1},\ldots,c_{i_p})$ where $\{c_{i_1},\ldots,c_{i_p}\}$ is a maximal subset of $\{c_1,\ldots,c_r\}$ such that each $c_{i_k}$ satisfies the condition \eqref{eq stop condition 0} 
\item \label{condition stop} If $U_n V_n=0$ mod $K'_n$, then return $(\Sym^{\bullet} M,K'_n,U_n,V_n)$
\item Else $n:=n+1$ and go to Step \eqref{goto}
\end{enumerate}
\end{algo}

\begin{remark} \label{remark difficult steps}---
\begin{itemize}
\item Steps \eqref{hard}, \eqref{rey2}, and \eqref{rey3} require the computation of Reynolds operators, which is done by implementing Algorithm 4.5.19 from \cite{DK}.
\item Steps \eqref{debut tang} to \eqref{fin tang} implement the procedure described in section \ref{first order} to compute the tangent space $T^1=T_{[X]} \HH$. This part of the algorithm can be implemented independently of the rest if one is only interested  in the tangent space.
\item As mentioned earlier, we have an explicit upper bound for $n$ (see section \ref{subsec stop}), but we do not have such any bound for $m$ (see section \ref{first order}). In particular, we know that our algorithm has to stop, but we do not know its complexity. 
\item If there is a subgroup $ \Gm \subset \Aut^G(W)$ which acts on $P$ and stabilizes the ideal $I$ but does not satisfy Hypothesis \ref{hypo weights}, then   Algorithm \ref{algo main} can still be used to calculate the universal deformation up to a given order. However, in this case it might happen that the stop condition \eqref{condition stop}  is never satisfied.
\item Given only $P$ and $I$ with their $G$-action, our algorithm cannot decide whether Hypothesis \ref{hypo weights} holds. Hence, this part of the calculation has to be done by hand before applying our algorithm.
\end{itemize}
\end{remark}

\section{An application: the action of \texorpdfstring{$SO_3$}{SO3} on \texorpdfstring{$(k^3)^{\oplus 3}$}{(k3)3}}  \label{HclassSOngeneral}

\subsection{Setting and main result} \label{HclassSOngeneralset}

Let $V$, $V'$ be $3$-dimensional vector spaces. We take $G=SO(V)$, $H=GL(V')$, and $W=\Hom(V',V) \cong V^{\oplus 3}$. For all practical purposes, we identify $W$ with the space of $3\times 3$-matrices $k^{3\times 3}$. The group $G \times H$ acts on $W$ by:
\begin{equation}  \label{Haction}
(g,h).w:=g \circ w \circ h^{-1}
\end{equation}
for $w \in W$ and $(g,h) \in G \times H$. 
We fix once and for all the Hilbert function
\begin{equation} \label{functionh}
h_0:\Irr(G) \to \NN,\quad M \mapsto \dim(M).
\end{equation}
As in section \ref{sec IHS}, we denote by $\HH=\Hilb_{h_0}^G(W)$ the invariant Hilbert scheme corresponding to the triple $(G,W,h_0)$, and by $\gamma:\HH \to W/\!/G$ the Hilbert-Chow morphism. 
We will see in section \ref{quotientSO3} that $W/\!/G$ is an affine cone whose vertex, denoted by $0$, is the only closed orbit for the $H$-action. Hence, it is natural to ask what the projective scheme $\gamma^{-1}(0)$ looks like. By Proposition \ref{chow}, the point $[X] \in \HH$ belongs to $\gamma^{-1}(0)$ if and only if the ideal $I_X$ contains the homogeneous $G$-invariants of positive degree of $k[W]$. In particular, as $h_0(V_0)=\dim(V_0)=1$, where $V_0$ denotes the trivial representation of $G$, the subset of $\HH(k)$ corresponding to homogeneous ideals of $k[W]$ is contained in $\gamma^{-1}(0)$ seen as a set.

The main result of the section \ref{HclassSOngeneral} is the following one:

\begin{theorem}  \label{casSO3}
Let $G$, $W$, and $\HH$ be as defined above. Then:
\begin{enumerate}
\item The invariant Hilbert scheme $\HH$ is reduced, connected, and has exactly two irreducible components: 
\begin{itemize}
\item the main component $\HHm$, which is smooth of dimension $6$; and
\item another component $\HH'$ of dimension $5$, whose singular locus has dimension $2$.
\end{itemize}
The intersection $\HHm \cap \HH'$ is irreducible of dimension 4, and its singular locus has dimension 2.
\item  The scheme-theoretic fiber $\gamma^{-1}(0)$ of the Hilbert-Chow morphism $\gamma:\HH \to W/\!/G$ over the vertex of the affine cone $W/\!/G$ is non-reduced, connected, contained in $\HHm$, and has exactly two irreducible components:
\begin{itemize}
\item a component of dimension 5, smooth, which is exactly the subset of $\HH(k)$ formed by homogeneous ideals of $k[W]$; and
\item a component of dimension 3, whose singular locus has dimension 2. 
\end{itemize}
\end{enumerate} 
\end{theorem}

\begin{remark}
The extra component $\HH'$ is not contained in $\gamma^{-1}(0)$. In particular, $\HH'$ is not formed by homogeneous ideals of $k[W]$. We emphasize this fact because for $G=O_3$ or $GL_3$, we will see in section \ref{two others} that the extra irreducible component of $\HH$ which appears is formed only by homogeneous ideals of $k[W]$.
\end{remark}

\begin{corollary}
With the notation above, the restriction of the Hilbert-Chow morphism to the main component $\gamma:\HHm \to W/\!/G$ is a resolution of singularities.
\end{corollary}

First of all, in section \ref{quotientSO3}, we will study the quotient morphism $\nu:W \to W/\!/G$ and see that the general fibers of $\nu$ are isomorphic to $G$. This will imply that the Hilbert-Chow morphism $\gamma:\HH \to W/\!/G$ is an isomorphism over a non-empty open subset; see Corollary \ref{cor gamma iso}. Next, in section \ref{subsec fix so3}, we will determine the only two fixed points of $\HH$ for the action of a Borel subgroup of $H$, and show that they live in the main component of $\HH$. In particular, this gives the connectedness of $\HH$ by Lemma \ref{closedpoints}. Then, we will determine the tangent spaces to $\HH$ at each of these fixed points and see that one of the fixed points is smooth while the other, say $[X_0]$, is singular. Finally, in section \ref{proofSO3}, we will apply our algorithm to the ideal of $X_0 \subset W$ and  obtain an affine open neighborhood $U \subset \HH$ of $[X_0],$ as well as the restriction of the universal family over $U$, and finish the proof of Theorem \ref{casSO3}. 

Sections \ref{quotientSO3} and \ref{subsec fix so3} are mainly extracted from \cite[\S3.2] {Terp}. However, section \ref{proofSO3}, which is by far the most important part of section \ref{HclassSOngeneral}, is an original work.

\subsection{The quotient morphism}  \label{quotientSO3}

The quotient morphism $\nu: \ W \to W/\!/G$ can be explicitly described as follows. Consider the morphism
\begin{equation}\label{eq quot so3} 
\begin{array}{lrcl}
 \mu:  & W   & \to  & S^2(V'^*) \times \Aff \\
        & Q  & \mapsto      &  (\leftexp{t}{Q}Q,\ \det(Q)) 
\end{array}
\end{equation}
where $S^2(V'^*)$ denotes the symmetric square of $V'^*$, $\det$ is the determinant, and $\leftexp{t}{Q}$ is the transpose of the matrix $Q$. Recall that the action of $H=GL(V')$ on $W$ induces an action of $H$ on $W/\!/G$ such that $\nu$ is $H$-equivariant. 

\begin{proposition}\label{prop quotient morphism}
The morphism $\mu$ factors as the composition of the quotient morphism $\nu$ and a closed immersion such that the following holds.
\begin{enumerate}
\item\label{item quotient} The quotient $W/\!/G$ identifies with the closed subvariety $$\left\{ (Q,x) \in  S^2(V'^*) \times \Aff  \  \mid  \  \det(Q)=x^2 \right\}$$ of $S^2(V'^*) \times \Aff$.
\item The singular locus of $W/\!/G$ is $\{(Q,x) \in W/\!/G \ |\ \rg(Q) \leq 1\}$. 
\item The variety $W/\!/G$ decomposes into  $4$ orbits for $H$ which are given by
$$U_i:=\left\{  (Q,x) \in  W/\!/G \  \mid \ \rg(Q) = i \right\}, \hspace{4mm} \text{for $i=0,\ldots,3$}.$$
The closures of these orbits are nested in the following way: 
$$ \{0\}=\overline{U_0} \subset \overline{U_1} \subset \overline{U_2} \subset \overline{U_3}=W/\!/G.$$ 
\item\label{item platitude} The flat locus of the quotient morphism $\nu:W \to W/\!/G$ is  $U_3 \cup U_2 $. 
\end{enumerate}
\end{proposition}

\begin{proof}
The morphism $\mu$ is $G$-invariant, so by the universal property of the categorical quotient, $\mu$ factors through $\nu$. By the First Fundamental Theorem for $SO(V)$ (see \cite[\S11.2.1]{Pro}) the coordinates of the image of $\mu$ generate the ring of invariants, and thus $W/\!/G$ embeds into $S^2(V'^*) \times \Aff$.  Clearly, $\det(Q)=x^2$ is the only relation, and the first statement follows.

The second assertion is easily verified using the Jacobian criterion, and the third statement is linear algebra.

For the last claim note that $\nu$ is flat over a non-empty open subset of $W/\!/G$, hence over $U_3$. Now one may check (see \cite[\S3.2.1] {Terp}) that the dimension of a fiber of $\nu$ over $U_i$ is 3 if $i \in \{2,3\}$, and is at least 4 if $i \in \{0,1\}$. The result then follows from \cite[Exercise III.10.9]{Ha}, because equidimensional morphisms between regular schemes are flat.
\end{proof}

\begin{remark}
A study of the quotient morphism $\nu: \ W \to W/\!/G$ for $V$ and $V'$ of arbitrary dimensions can be found in \cite[\S3.2.1] {Terp}. 
\end{remark}

\begin{corollary} \label{fibgeneSOn}
The general fibers of the quotient morphism $\nu:W \to W/\!/G$ are isomorphic to $G$. 
In particular, the Hilbert function $h_0$ of the general fibers of $\nu$ is given by \eqref{functionh}.
\end{corollary}

\begin{proof}
We fix bases for $V$ and $V'$, and we identify $W=\Hom(V',V)$ with the space of $3\times 3$-matrices. Denoting $\id$ the identity map, we have $\nu(\id)=(\id,1) \in U_3$. The stabilizer of $\id$ in $G=SO(V)$ is trivial, hence $\nu^{-1}((\id,1))$ contains a closed $G$-orbit isomorphic to $G$. As a fiber of $\nu$ always contains a unique closed $G$-orbit, and as $\dim(G)=3$ is also the dimension of the general fibers of $\nu$, we must have $\nu^{-1}((\id,1)) \cong G$. Recalling that
$$ k[G] \cong \bigoplus_{M \in \Irr(G)} M \otimes M^*$$
as a $(G \times G)$-module (see e.g. \cite[\S7.3.1 Theorem]{Pro}), we get that $h_0(M)=\dim(M)$.
\end{proof}

Combining Proposition \ref{prop quotient morphism} \eqref{item platitude} with Proposition \ref{chow} we find

\begin{corollary}\label{cor gamma iso}
The Hilbert-Chow morphism  $\gamma : \HH\to W/\!/G$ is an isomorphism over  $U_3 \cup U_2 $.
\end{corollary}

By definition, the main component of $\HH$ is $\HHm=\overline{\gamma^{-1}(U_2 \cup U_3)}$. It follows that the extra components of $\HH$, if any, have to be contained in $\gamma^{-1}(\overline{U_1})$.

\subsection{Fixed points for the action of a Borel subgroup} \label{subsec fix so3}

By \cite[\S10.4]{FH}, there is an isomorphism of algebraic groups $SO_3 \cong {PSL}_2$, where ${PSL}_2:={SL}_2/\{\pm Id\}$. The irreducible representations of $SL_2$ are parametrized by nonnegative integers: 
$d \in \NN \leftrightarrow V_d$, where $V_d:={k[x,y]}_d$ is the space of homogeneous polynomials of degree $d$. In particular, $\dim(V_d)=d+1$. The irreducible representations of $G \cong SO_3$ are thus parametrized by even nonnegative integers. The trivial representation is $V_0$, and the defining representation is $V_2$. We recall that one can easily decompose tensor products of irreducible representations of $G$ using the Clebsch-Gordan formula (\cite[Exercise 11.11]{FH}).

We have 
\begin{align}
&{k[W]}_1 \cong V' \otimes V_2; \text{ and } \label{kW1SO3}\\
&{k[W]}_2 \cong (S^2V' \otimes (V_4 \oplus V_0)) \oplus ({\Lambda}^2V' \otimes V_2); \label{kW2SO3} 
\end{align}
as $(G \times H)$-modules. 

We fix a Borel subgroup $B \subset H$. For explicit calculations, we agree to take $B$ to be the subgroup of upper triangular matrices. Recall that every irreducible $H$-module contains a unique $B$-stable line. We denote by $D_1 \subset V'$ and by  $D_2\subset S^2V'$ these unique $B$-stable lines.

\begin{definition}\label{definition i und iprime}
Let $I\subset k[W]$ be the ideal generated by  
$$(D_2 \otimes V_4) \oplus (S^2V' \otimes V_0) \oplus ({\Lambda}^2 V' \otimes V_2) \subset {k[W]}_2,$$
and let $I'\subset k[W]$ be the ideal generated by  
$$D_1 \otimes V_2 \subset {k[W]}_1\quad \textrm{and} \quad  (S^2V' \otimes V_0) \subset {k[W]}_2.$$
Moreover, let $X_0$ and $X'_0$ be the closed subschemes of $W$ defined by the ideals $I$ and~$I'$ respectively. 
\end{definition}

Note that the ideals $I$ and $I'$ are homogeneous and $(B \times G)$-stable. As explained in section \ref{subsec fix borel}, the first step to determine the global structure of $\HH$ is to determine the $B$-fixed points. The next result was shown in \cite {Terp} using representation theory of $G$ and $B$.  

\begin{theorem} \label{pointfixeborel36}  \emph{(\cite[Th�or�me 3.2.12] {Terp})}
The ideals $I$ and $I'$ are the only two fixed points of $\HH$ for the action of the Borel subgroup $B \subset H$.
\end{theorem}

\begin{remark}
One may check that $S:=\Stab_{H}(I')=\Stab_{H}(I)$ is the parabolic subgroup of $H$ stabilizing the line $D_1 \subset V'$. Hence, $S$ is maximal, of dimension 7, and the two closed $H$-orbits of $\HH$ are both isomorphic to $\PP(V')$.
\end{remark}

Following the strategy given in section \ref{subsec fix borel}, we find

\begin{proposition} \label{Hconnexe4}
The two fixed points of $\HH$ for the action of the Borel subgroup $B \subset H$ belong to the main component $\HHm$. In particular, $\HH$ is connected.
\end{proposition}

\begin{proof}
We will construct flat families $p:\ZZ\to \Aff$ of $G$-stable closed subschemes of $W$ such that the zero-fiber $\ZZ_0$ has ideal $I$ respectively $I'$, and such that for every $t \neq 0$, the fiber of $p$ over $t$ corresponds to a point of $\HHm$.

By Corollary \ref{cor gamma iso}, the Hilbert-Chow morphism $\gamma:\HH\to W/\!/G$ is an isomorphism over $U_3$.
Consequently, the unique $[X_{\id}] \in \HH$ such that $\gamma([X_{\id}])=(\id,1)$ is contained in $\HHm$. 
Let $\theta: \Gm\to B$ be a one-parameter subgroup such that $\theta(t)(\id,1)$ goes to $0\in W/\!/G$ when $t\to 0$. Then, by properness of the Hilbert-Chow morphism, there is a unique flat family $p:\ZZ(\theta)\to \Aff$ of $G$-stable closed subschemes of $W$ such that $\ZZ(\theta)_t:=p^{-1}(t)$ is given by $\theta(t)\cdot X_{\id}$ for $t\neq 0$. The proposition now follows from Lemma \ref{lemma connect} below,  which will ensure the existence of $\theta$ and $\theta'$ such that $\ZZ(\theta)_0 = X_0$ and $\ZZ(\theta')_0 = X_0'$. As $\nu$ is equivariant, we see that the generic fiber of $p$ is sent to a point in $\HHm$, and hence every fiber does so.
The connectedness is then a direct consequence of Lemma \ref{closedpoints}. 
\end{proof}

We now introduce some notation. 
Let $\theta: \Gm\to B$ be a one-parameter subgroup. For $P\in k[W]$ let $m\in \ZZZ$ be the order in $t$ of $\theta(t).P$ at $t=0$, and let $$\tilde{P}(t):=t^{-m} (\theta(t).P).$$
Note that $\tilde{P}(t) \in k[W][t]$ by construction. Denote by $L \subset k[W]$ the ideal of the scheme $X_{\id}$, and by $L_t \subset k[W][t]$ the ideal of  $\theta(t)\cdot X_{\id}$ for $t\neq 0$. By \cite[Exercise 15.25]{Ei}, if $\{P_1,\ldots,P_r\}$ is a Gr�bner basis of $L$, then  
$$L_t=(\tilde{P_1}(t),\ldots,\tilde{P_r}(t)).$$
In particular, we have $L_1=L$ and $L_0=(\tilde{P_1}(0),\ldots,\tilde{P_r}(0))$ is the ideal of the ``flat degeneration'' $\ZZ_0$ constructed in the proof of Proposition \ref{Hconnexe4}. 

The proof of the next lemma is obtained by conducting the above procedure with a computer algebra system; see \cite[\S 3.2.2] {Terp} for details.

\begin{lemma}\label{lemma connect}
For $n=(n_1,n_2,n_3) \in {\ZZZ}^3$, we denote  by $\theta_n$ the one-parameter subgroup of $B$ defined by
\begin{equation*}  
\begin{array}{ccccc}
\theta_{n} & : & \Gm & \to & B \\
& & t & \mapsto & \begin{bmatrix} t^{n_1} & 0 & 0 \\ 0 & t^{n_2}& 0 \\ 0& 0& t^{n_3}  \end{bmatrix}
\end{array}
\end{equation*}
Then, with the above notation, we obtain the following limit ideals: $L_0=I'$ for $n=(-3,-1,-1)$, and $L_0=I$ for $n=(-3,-2,-2)$.\\
\end{lemma}

Next, using the method described in section \ref{first order} to compute the dimension of the tangent space of $\HH$, we obtain the following:

\begin{proposition} \label{dimTangentSO3}
The tangent space $T_{[X_0]} \HH$ is 8-dimensional and $T_{[X'_0]} \HH$ is 6-dimensional. 
In particular, $\HHm$ is smooth at $[X'_0]$ and singular at $[X_0]$. \qed
\end{proposition}

Let us mention that the dimension of $T_{[X_0]} \HH$ and $T_{[X'_0]} \HH$, and even the $B$-module structure, has been calculated mostly by hand in \cite[\S 3.2.2]{Terp}. The innovation here is that we are now able to calculate a $k$-basis algorithmically using a computer algebra system.

\subsection{Proof of the main result}  \label{proofSO3}

In this section, we prove Theorem \ref{casSO3}. We already know that the invariant Hilbert scheme $\HH$ is connected by Proposition \ref{Hconnexe4}.
One may check that $B$ contains a multiplicative subgroup $\Gm$ such that Hypothesis \ref{hypo weights}  is satisfied for $[X_0] \in \HH$. Applying our algorithm from section \ref{section algorithm} to the ideal $I$ of $X_0 \subset W$, we obtain the existence of an affine open neighborhood $U \subset \HH$ containing $[X_0]$ such that 
$$U \cong \Spec \left( \frac{k[t_1,\ldots,t_8]}{K} \right),$$
where $K$ is the ideal generated by the four elements:
$$
    \begin{array}{l}
405 t_2 t_5+810 t_1 t_6+36 t_3 t_6 t_7-54 t_3 t_5 t_8-90 t_2 t_7 t_8-90 t_1 t_8^2+8 t_3 t_7 t_8^2;\\
810 t_2 t_4+405 t_1 t_5+18 t_3 t_5 t_7-90 t_2 t_7^2-108 t_3 t_4 t_8-90 t_1 t_7 t_8+8 t_3 t_7^2 t_8;\\
15 t_2 t_3-2 t_3^2 t_8; \textrm{ and}\\ 
45 t_1 t_3+2 t_3^2 t_7. 
    \end{array}
$$
One may check that $K$ is radical, and that the prime decomposition of $K$ is given by $K=K_1 \cap K_2$, where 
$$
    \begin{array}{ll}
K_1= &(2 t_3 t_8 - 15 t_2, 2 t_3 t_7 + 45 t_1, t_2 t_7 + 3 t_1 t_8); \textrm{ and}\\
K_2= &(t_3, 2 t_2 t_7 t_8 + 2 t_1 t_8  - 9 t_2 t_5 - 18 t_1 t_6, 4 t_6 t_7  - 4 t_5 t_7 t_8 + 4 t_4 t_8  + 9 t_5  - 36 t_4 t_6,\\
      &2 t_2 t_7  + 2 t_1 t_7 t_8 - 18 t_2 t_4 - 9 t_1 t_5, t_2 t_5 t_7 + 2 t_1 t_2 t_6 t_7 - 2 t_2 t_4 t_8 - t_1 t_2 t_5 t_8).
    \end{array}
$$

Hence, $U$ is the union of two irreducible components: $C_1=Z(K_1)$, which is 6-dimensional and smooth, and $C_2=Z(K_2)$, which is 5-dimensional and whose singular locus is 2-dimensional. Moreover, $C_1 \cap C_2$ is reduced, irreducible, 4-dimensional, and its singular locus is 2-dimensional. Let us note that, as $\dim(\HHm)=6$ and $[X_0] \in \HHm$, we must have $C_1=U \cap \HHm$.

Now suppose that $\HH$ is non-reduced, then the support of the non-reduced part of $\HH$, say $F$, is a $H$-stable closed subset of $\HH$. By Lemma \ref{closedpoints}, $F$ has to contain a fixed point for the action of the Borel subgroup $B \subset H$. However, we already know that the only two $B$-fixed points do not belong to $F$; indeed, $\HH$ is reduced around $[X_0]$, since $K$ is radical, and $[X'_0]$ is a smooth point of $\HH$. Hence, $F$ is empty, that is, $\HH$ is reduced.

Arguing in the same way, we easily prove that the singular locus of $\HHm$ is empty, that is, $\HHm$ is smooth, that $\HH$ does not have any further irreducible components, and that $\HHm \cap \HH'$ has the same geometrical properties as $C_1 \cap C_2$.

Let us now prove the part (2) of Theorem \ref{casSO3}. First, it follows from Lemma \ref{closedpoints} and Proposition \ref{Hconnexe4} that $\gamma^{-1}(0)$ is the union of at most two connected components that both intersect the main component $\HHm$. However, as $\gamma$ is proper and $W/\!/G$ is normal  (see e.g. \cite[\S 3.2, Theorem 2]{SB}), it follows from Zariski's Main Theorem (\cite[Corollary 11.4]{Ha}) that $\HHm \cap \gamma^{-1}(0)$ is connected, and thus $\gamma^{-1}(0)$ is connected. 

Now our algorithm gives not only the ideal $K$ mentioned above, but also an ideal $J \subset k[W] \otimes \frac{k[t_1,\ldots,t_8]}{K}$ such that the natural morphism
$$\Spec \left( \frac{k[W] \otimes \frac{k[t_1,\ldots,t_8]}{K}}{J} \right)  \to U \cong \Spec \left( \frac{k[t_1,\ldots,t_8]}{K} \right)$$
is the restriction over $U$ of the universal family $\UU \to \HH$ mentioned in section \ref{sec IHS}. Looking closely to this family, we can write down explicitly the restriction on $U$ of the Hilbert-Chow morphism $\gamma:\HH \to W/\!/G$. We obtain that the ideal $K_0$ of $U\cap \gamma^{-1}(0)$ is generated by the following seven elements:
$$
    \begin{array}{l}
405 t_2 t_5+810 t_1 t_6+36 t_3 t_6 t_7-54 t_3 t_5 t_8-90 t_2 t_7 t_8-90 t_1 t_8^2+8 t_3 t_7 t_8^2;\\ 
810 t_2 t_4+405 t_1 t_5+18 t_3 t_5 t_7-90 t_2 t_7^2-108 t_3 t_4 t_8-90 t_1 t_7 t_8+8 t_3 t_7^2 t_8; \\
2025 t_1^2-2025 t_3^2 t_4+221 t_3^2 t_7^2; \\
1350 t_1 t_2+675 t_3^2 t_5-142 t_3^2 t_7 t_8; \\  
75 t_2^2-75 t_3^2 t_6+7 t_3^2 t_8^2;\\ 
15 t_2 t_3-2 t_3^2 t_8; \text{ and } \\                            
45 t_1 t_3+2 t_3^2 t_7. 
     \end{array}              
$$
One may check that the ideal $K_0$ is not radical, that is, $\gamma^{-1}(0)$ is non-reduced, and that $U_0:=U \cap \gamma^{-1}(0)$ is the union of two irreducible components. Now, arguing with $K_0$ as before with $K$, we easily prove the second part of Theorem \ref{casSO3}. Finally, it follows from a careful study of the restriction  of the universal family $\UU \to \HH$ to  $U_0$ that one of the two irreducible components of $\gamma^{-1}(0)$ is exactly the subset of $\HH(k)$ formed by homogeneous ideals of $k[W]$.

\section{Two other applications}  \label{two others}

In section \ref{section O3} we will determine the structure of the invariant Hilbert scheme for the action of $O_3$ on several copies of the defining representation, and in section \ref{section GL3} we will do the same for the action of $GL_3$ on classical representations. In both cases, we will see that there is an extra component, besides the main component, formed only by homogeneous ideals. Recall from Theorem \ref{casSO3} that for $SO_3$ the extra component also contained non-homogeneous ideals. On the other hand, we will see that for $GL_3$ the extra component has bigger dimension than the main component unlike for $SO_3$ or $O_3$. 

Hence, it appears that the geometrical properties of the invariant Hilbert scheme can be very different from one case to another, whence the necessity to determine many more examples in the future.
 
\subsection{Case of \texorpdfstring{$O_3$}{O3} acting on \texorpdfstring{$(k^3)^{\oplus n}$}{(k3)n}}  \label{section O3}

Let $V$ be a 3-dimensional vector space, and let $G=O(V)$ be the orthogonal group. For $n \geq 3$, we consider $W=V^{\oplus n}$ with the induced $G$-action. If we identify $W$ with the space of matrices $k^{3\times 3}$, then it follows from the First Fundamental Theorem for $O(V)$ (see \cite[\S11.2.1]{Pro}) that the quotient morphism is given by:
$$\begin{array}{cccc} 
\nu: \ &W &\to &W/\!/G \\
       &Q & \mapsto  & \leftexp{t}{Q}Q
\end{array}$$
where $\leftexp{t}{Q}$ denotes the transpose of the matrix $Q$. In particular 
$$W/\!/G \cong \{ Q \in k^{n\times n} \ |\ Q=\leftexp{t}{Q} \text{ and } \rk(Q)\leq 3\}$$ 
is a symmetric determinantal variety. The quotient morphism $\nu$ was studied for $\dim(V)$ and $n$ arbitrary in \cite[\S3.1.1] {Terp}. One easily checks that the general fibers of $\nu$ are isomorphic to $G$; in particular, the Hilbert function of the general fibers of $\nu$ is 
\[
h_0:\Irr(G) \to \NN, \quad M \mapsto \dim(M).
\]

\begin{theorem}  \label{casO3}
Let $G$, $W$, and $h_0$ be as defined above. Then the invariant Hilbert scheme $\HH=\Hilb_{h_0}^G(W)$ is connected and has at least two irreducible components: 
\begin{itemize}
\item the main component $\HHm$ of dimension $3n-3$; and
\item another component $\HH'$ of dimension $3n-4$, whose points correspond to homogeneous ideals of $k[W]$.
\end{itemize}
Moreover the intersection $\HHm \cap \HH'$ has dimension $3n-5$.
\end{theorem}

\begin{proof}
The proof is very similar to the one of Theorem \ref{casSO3} and thus we just give an outline:
\begin{enumerate}
\item We use the reduction principle obtained in \cite[\S1.5.1] {Terp} to reduce from the case $n\geq 3$ to the case $n=3$.
\item Denoting $V'=k^3$, we identify $W$ with $\Hom(V',V)$ on which $H=GL(V')$ acts by $h.w=w \circ h^{-1}$. Then we fix a Borel subgroup $B \subset H$ and we show (\cite[Theorem 3.1.29] {Terp}) that $\HH$ admits only two $B$-fixed points $[X_1]$ and $[X_2]$.
\item Connectedness is obtained  by showing, as in Proposition \ref{Hconnexe4}, that the two $B$-fixed points belong to the main component $\HHm$. This has been done for one of them, say $[X_1]$, in \cite[Proposition 3.1.33]{Terp}. For the other one, say $[X_2]$, it is more involved to find a good starting point for the one-parameter subgroup of diagonal matrices. We found it thanks to our algorithm as explained in Step \eqref{thm O3 last step}.
\item We show, using the method described in section \ref{first order}, that the dimension of the tangent spaces $T_{[X_1]}\HH$ and $T_{[X_2]}\HH$ is 7 in both cases. For $[X_1]$, we can find a multiplicative subgroup $\Gm \subset H$ such that Hypothesis \ref{hypo weights}  holds. Then we apply our algorithm from  section \ref{section algorithm}, and we obtain the existence of an affine neighborhood $U \subset \HH$ of $[X_1]$ such that
$$ U \cong \Spec \left( \frac{k[t_1,\ldots,t_7]}{(t_2t_4-t_2 t_5, t_1 t_4-t_1 t_5)} \right) .$$
Hence, the invariant Hilbert scheme has at least two irreducible components which locally on $U$ are given by $C_1=Z(t_4-t_5)$ and $C_2=Z(t_1,t_2)$, and whose intersection is $Z(t_1,t_2,t_4-t_5)$. As $\HHm$ is 6-dimensional and $[X_1] \in \HHm$, we must have $C_1=\HHm \cap U$. Again, the study of the restriction of the universal family $\UU \to \HH$ to $U$ entails that points of the second component correspond exactly to homogeneous ideals of $k[W]$.
\item \label{thm O3 last step} For $[X_2]$ one can show that there is no one-parameter subgroup of $H$ with strictly positive weights. This is because the maximal torus $T\subset B$ has three vectors of weight $0$ in $(T^1)^{\vee}$. So whatever combination of exponents we take for a subgroup of diagonal matrices, these three vectors will always have weight zero. Finally, as all one-parameter subgroups are obtained from the diagonal ones by conjugation, this holds in general. However, we can find a subgroup $\Gm \subset B$ acting on a $4$-dimensional subspace of $(T^1)^{\vee}$  with strictly positive weights. Our algorithm, with the tangent space $(T^1)^{\vee}$ replaced by this four dimensional subspace, produces a family of $G$-stable closed subschemes of $W$ parametrized by $\AAA^4$. An easy analysis of this family gives that its  general member lives in $\HHm$. In other words, there is a morphism $\AAA^4\to \HHm$ sending $0$ to $[X_2]$; in particular, $\HH$ is connected. Although it is not needed in the proof, Corollary \ref{cor open immersion} actually shows that $\AAA^4\to \HHm$ is an immersion. 
\end{enumerate}
\end{proof}

\begin{remark}
In the setting of Theorem \ref{casO3}, it would have been pleasant to determine whether $\HH$ is reduced, whether there are other irreducible components, and also whether the main component $\HHm$ is smooth. Unfortunately, there is no multiplicative subgroup $\Gm \subset H$ such that Hypothesis \ref{hypo weights}  is satisfied for the $B$-fixed point $[X_2]$ as explained in the proof of Theorem \ref{casO3}.
\end{remark}

\subsection{Case of \texorpdfstring{$GL_3$}{GL3} acting on \texorpdfstring{$(k^3)^{\oplus n_1} \oplus (k^{3*})^{\oplus n_2}$}{(k3)n1+(k3*)n2}}  \label{section GL3}

Let $V$ be a 3-dimensional vector space, let $n_1,n_2 \geq 3$, and let $W=V^{\oplus n_1} \oplus V^{* \oplus n_2}$ on which $G=GL(V)$ acts naturally. If we identify $W$ with $k^{3 \times n_1} \times k^{n_2\times 3}$, then it follows from the First Fundamental Theorem for $GL(V)$ (see \cite[\S9.1.4]{Pro}) that the quotient morphism is given by: 
$$\begin{array}{cccc} 
\nu: \ &W &\to &W/\!/G \\
       &(Q_1,Q_2) & \mapsto  & Q_2 Q_1
\end{array}$$
In particular 
$$W/\!/G=(k^{n_2\times n_1})^{\leq 3}:=\{ Q \in k^{n_2\times n_1} \ |\ \rk(Q) \leq 3\}$$ 
is a determinantal variety, which is smooth if $\min(n_1,n_2)=3$, and whose singular locus is $(k^{n_2\times n_1})^{\leq 2}$ else. Let us mention that the quotient morphism $\nu$ was studied for $\dim(V),n_1, n_2$ arbitrary in \cite[\S2.1.1] {Terp}. One may check that the general fibers of $\nu$ are isomorphic to $G$, and thus the Hilbert function of the general fibers of $\nu$ is 
\[
h_0:\Irr(G) \to \NN, \quad M \mapsto \dim(M).
\]
Let us denote $\HH=\Hilb_{h_0}^G(W)$. It follows from the results of \cite[section 2.1.3] {Terp} that the subset of $\HH(k)$ corresponding to homogeneous ideals of $k[W]$ coincides with the zero-fiber $\gamma^{-1}(0)$ (as a set) of the Hilbert-Chow morphism $\gamma: \ \HH \to W/\!/G$.

\begin{theorem}  \label{casGL3}
Let $G$, $W$, and $h_0$ be as defined above. Then:
\begin{enumerate}
\item The invariant Hilbert scheme $\HH=\Hilb_{h_0}^G(W)$ is reduced, connected, and has exactly two irreducible components: 
\begin{itemize}
\item the main component $\HHm$, which is smooth of dimension $3n_1+3n_2-9$; and
\item another component $\HH'$, which is smooth of dimension $3n_1+3n_2-8$, and formed by homogeneous ideals of $k[W]$.
\end{itemize}
The intersection of these two components $\HHm \cap \HH'$ is irreducible, smooth, and has dimension $3n_1+3n_2-11$; in particular, they intersect transversally.
\item  The scheme-theoretic fiber $\gamma^{-1}(0)$ of the Hilbert-Chow morphism $\gamma:\HH \to W/\!/G$ is reduced, connected, and has two irreducible components: the smooth component $\HH'$ described above and a smooth hypersurface contained in $\HHm$.
\end{enumerate} 
\end{theorem}

\begin{proof}
The proof is very similar to the one of Theorem \ref{casSO3} and thus, as we did for Theorem \ref{casO3}, we just give an outline:
\begin{enumerate}
\item We use the reduction principle obtained in \cite[1.5.1] {Terp} to reduce from the case $n_1, n_2 \geq 3$ to the case $n_1=n_2=3$.
\item Denoting $V_1=V_2=k^3$, we identify $W$ with $\Hom(V_1,V) \times \Hom(V,V_2)$ on which $H=GL(V_1) \times GL(V_2)$ acts by 
$$(h_1,h_2).(w_1,w_2)=(w_1 \circ h_1^{-1}, h_2 \circ w_2)$$ 
for all $(h_1,h_2) \in H$ and $(w_1,w_2)\in W$.
Let us fix a Borel subgroup $B \subset H$. It was shown in \cite[Theorem 2.1.49] {Terp} that $\HH$ has only one $B$-fixed point $[X]$. In particular, $\HH$ is connected.
\item It was shown in \cite[Proposition 2.1.53] {Terp} that the dimension of the tangent space $T_{[X]}\HH$ is 12. We easily find a multiplicative subgroup $\Gm \subset H$ such that Hypothesis \ref{hypo weights}  holds, and we apply the algorithm described in section \ref{section algorithm}. We obtain the existence of an affine neighborhood $U \subset \HH$ of $[X]$ such that 
$$ U \cong \Spec \left( \frac{k[t_1,\ldots,t_{12}]}{K} \right),$$
where $K$ is the ideal generated by the six elements:
$$
    \begin{array}{l}
t_2 t_3-t_2 t_5-t_2 t_8 t_{11},\\
t_1 t_3-t_1 t_5-t_1 t_8 t_{11},\\
t_2 t_6, t_1 t_6, t_2 t_4, t_1 t_4.
     \end{array}              
$$
Hence, $U$ is the union of the two irreducible components $C_1=Z(t_4,t_6,t_8 t_{11}-t_3+t_5)$ and $C_2=Z(t_1,t_2)$ whose intersection is the irreducible variety $Z(t_1,t_2,t_4,t_6,t_8 t_{11}-t_3+t_5)$. As $\HHm$ is 9-dimensional, we must have $C_1=\HHm \cap U$. An analysis of the universal family $\UU \to \HH$, restricted to $U$, shows that points of the second component $C_2$ correspond to homogeneous ideals of $k[W]$. 
\item Arguing as for the proof of Theorem \ref{casSO3}, we get that: 
\begin{itemize} 
\item the reducibility of $U$ implies the reducibility of $\HH$;
\item the smoothness of $U \cap \HHm$ respectively of $U \cap \HH'$, implies the smoothness of $\HHm$ respectively of $\HH'$;
\item the geometrical properties of $\HHm \cap \HH'$ are the same as that of $C_1 \cap C_2$; and
\item $\HH$ has no more component.
\end{itemize}
\item The proof of the second part of Theorem \ref{casGL3}, which concerns the scheme-theoretic fiber $\gamma^{-1}(0)$, is analogous to the proof of the first part concerning the whole invariant Hilbert scheme. It suffices to consider $U \cap \gamma^{-1}(0)$ instead of $U$, the ideal $K_0$ of $U \cap \gamma^{-1}(0)$ being $K_0=K+(t_1)$. 
\end{enumerate}
\end{proof}

\begin{corollary}
Let $G$, $W$, and $\HH$ be as defined above. If $n_1,n_2>3$, then the restriction of the Hilbert-Chow morphism to the main component $\gamma:\HHm \to W/\!/G$ is a resolution of singularities.
\end{corollary}

\begin{remark}
Let $B$ be the Borel group mentioned in the proof of Theorem \ref{casGL3}. The varieties $\HHm$ and $W/\!/G$ are normal, and contain an open orbit for the action of $B$, that is, $\HHm$ and $W/\!/G$ are \textit{spherical varieties}. See for instance \cite{Tim} for an introduction to the theory of spherical varieties. Using this theory,  one can show that the Hilbert-Chow morphism $\gamma:\HHm \to W/\!/G$ identifies with the composition of blows-up $f_2 \circ f_1 \circ f_0$, where
\begin{itemize}
\item $f_0$ is the blow-up of $\{0\} \in W/\!/G$; and
\item $f_i$ is the blow-up of the strict transform of $(k^{n_2\times n_1})^{\leq i}$. 
\end{itemize}
Similar results with other $(W,G)$ can be found in \cite[Theorem]{Ter1}. 
\end{remark}

\bibliographystyle{alpha}

\end{document}